\documentclass[10pt,reqno,fleqn]{amsart}

\usepackage{amsmath,amscd,amssymb,amsfonts,amsthm,courier,relsize,bm}
\usepackage{hyperref,enumitem,mathrsfs,slashed,mathtools, extarrows}
\usepackage[bottom]{footmisc}
\usepackage{xcolor}  

\usepackage[bitstream-charter]{mathdesign} 
\usepackage[T1]{fontenc}

\hypersetup{
    colorlinks,
    linkcolor={red!50!black},
    citecolor={blue!70!black},
    urlcolor={blue!80!black}
}

\usepackage[cmtip,matrix,arrow]{xy}

\newtheorem{theorem}{Theorem}[section]
\newtheorem{corollary}[theorem]{Corollary}
\newtheorem{proposition}[theorem]{Proposition}

\newtheorem{lemma}[theorem]{Lemma}
\newtheorem*{keylemma}{Key Lemma}
\theoremstyle{definition}    

\newtheorem{definition}[theorem]{Definition}

\newtheorem{remark}[theorem]{Remark}

\newcommand{\pair}[2]{\langle #1, #2 \rangle}
\newcommand{\ignore}[1]{}

\newcommand{\ol}[1]{\overline{#1}}

\newcommand{\ti}[1]{\widetilde{#1}}
\newcommand{\wh}[1]{\widehat{#1}}
\newcommand{\st}[1]{\mathsf{#1}}
\newcommand{\scr}[1]{\mathscr{#1}}

\newcommand{\tn}[1]{\textnormal{#1}}
\renewcommand{\i}{{\mathrm{\bold{i}\,}}}

\def\g{\ensuremath{\mathfrak{g}}}

\def\S{\ensuremath{\mathfrak{S}}}
\def\f{\ensuremath{\mathfrak{f}}}

\def\F{\ensuremath{\mathcal{F}}}
\def\E{\ensuremath{\mathcal{E}}}

\def\L{\ensuremath{\mathcal{L}}}
\def\R{\ensuremath{\mathcal{R}}}
\def\C{\ensuremath{\mathcal{C}}}

\def\Cl{\ensuremath{\textnormal{Cl}}}
\def\Cliff{\ensuremath{\textnormal{Cliff}}}
\def\c{\ensuremath{\mathsf{c}}}

\def\bC{\ensuremath{\mathbb{C}}}
\def\bR{\ensuremath{\mathbb{R}}}

\def\bZ{\ensuremath{\mathbb{Z}}}

\def\End{\ensuremath{\textnormal{End}}}

\def\ker{\ensuremath{\textnormal{ker}}}
\def\id{\ensuremath{\textnormal{id}}}
\def\pt{\ensuremath{\textnormal{pt}}}

\def\index{\ensuremath{\textnormal{index}}}
\def\supp{\ensuremath{\textnormal{supp}}}
\def\dim{\ensuremath{\textnormal{dim}}}
\def\KK{\ensuremath{\textnormal{KK}}}

\def\K{\ensuremath{\textnormal{K}}}
\def\cl{\ensuremath{\textnormal{cl}}}
\def\ra{\ensuremath{\rangle}}
\def\la{\ensuremath{\langle}}

\begin{document}
\sloppy
\title{A KK-theoretic perspective on deformed Dirac operators}

\author{Yiannis Loizides, Rudy Rodsphon and Yanli Song}

\maketitle

\setlength{\skip\footins}{0.5cm}

\vspace{-0.4cm}

\setlength{\parindent}{0mm}

\begin{center}
\small \emph{Dedicated to Gennadi Kasparov for his $70^{\text{th}}$ birthday}
\end{center}

\begin{abstract}
We study the index theory of a class of perturbed Dirac operators on non-compact manifolds of the form $\mathsf{D}+\i\c(X)$, where $\c(X)$ is a Clifford multiplication operator by an orbital vector field with respect to the action of a compact Lie group. Our main result is that the index class of such an operator factors as a KK-product of certain KK-theory classes defined by $\mathsf{D}$ and $X$.  As a corollary we obtain the excision and cobordism-invariance properties first established by Braverman.  An index theorem of Braverman relates the index of $\mathsf{D}+\i\c(X)$ to the index of a transversally elliptic operator.  We explain how to deduce this theorem using a recent index theorem for transversally elliptic operators due to Kasparov.
\end{abstract}

\section*{Introduction} \setcounter{section}{0}

The present work studies, from the perspective of KK-theory, the index theory of a class of Dirac-type operators on non-compact manifolds. A Dirac operator $\st{D}$ on a non-compact manifold always determines a K-homology class $[\st{D}]$, but extracting a `numerical index' (in a possibly generalized sense) often requires additional ingredients, which can be either boundary conditions at infinity, or appropriate devices playing the role of compactness, relevant choices being dictated by the geometric situation at hand. In some cases, such a device can be a suitable perturbation of $\st{D}$.\\

A standard example on $\bR^n$ is the operator $d+d^\ast+\tn{ext}(x)+\tn{int}(x)$ acting on $L^2(\bR^n,\wedge T^\ast \bR^n)$, where $\tn{ext}(\cdot)$, $\tn{int}(\cdot)$ denote exterior and interior multiplication respectively.  Its square is a harmonic oscillator so that the operator has index one. On the KK-theoretic side, it is well-known that this operator represents the KK-product of the Bott/dual-Dirac and the Dirac elements, which is then the identity. Extensive generalizations of this calculation (in various forms) laid the foundations of important techniques used to prove KK-theoretic Poincar\'e duality results in index theory, or more broadly in much of the work done on the Baum-Connes conjecture. \\   

Another source of interesting examples is operators of Callias-type, introduced in \cite{Callias}. The perturbation here is a suitable `potential' $\Phi$, and the operator $\st{D}+\Phi$ is Fredholm. KK-product interpretations of the Fredholm index have been provided in \cite{Bunke}, in \cite{Kucerovsky} via unbounded KK-theory with recent improvements in \cite{KaadLesch}. Loosely speaking, the potential defines a $K$-theory class $[\Phi]$ on the manifold, and the index of $\st{D}+\Phi$ arises as the KK-product $[\Phi] \wh{\otimes}_{C_0(M)} [\st{D}]$.\\ 

In this article, we shall focus on a class of operators that we will call \emph{deformed Dirac operators}. Their study originates partly from \cite{TianZhang} and has been systematized by Braverman in \cite{Braverman2002}. These operators have found interesting applications, notably in the resolution of a conjecture of Vergne on the quantization commutes with reduction problem \cite{MaZhang}, and subsequent extensions of this work (e.g. \cite{HochsSong3}). \\

Let $M$ be a complete Riemannian manifold equipped with an isometric action $\rho$ of a compact Lie group $G$, let $E$ be a $G$-equivariant Clifford module bundle over $M$ and let $\st{D}$ be a Dirac operator acting on sections of $E$. In our setting the additional geometric data used to obtain a well-defined index is a $G$-equivariant map $\!\nu \colon M \to \g = \text{Lie}(G)$ such that the vector field $\bm{\!\nu}\colon m \in M \mapsto \rho_m(\!\nu(m))$ has a compact vanishing locus; Braverman \cite{Braverman2002} referred to such a map as a \emph{taming map}.  Deformed Dirac operators are then operators of the form $\st{D}_{f \!\nu} = \st{D} + \i f \c(\bm{\!\nu})$, where $f \in C^{\infty}(M)$ is a function satisfying a growth condition at infinity (see Section \ref{Section2}), and $\c(\bm{\!\nu})$ is Clifford multiplication.  A deformed Dirac operator has a well-defined equivariant index, similar to transversally elliptic operators (in the sense of Atiyah \cite{Atiyah}). We will come back to this analogy shortly.\\

An important technical consideration in studying $\st{D}_{f \!\nu}$ lies in the calculation of the commutator of $\st{D}$ with the perturbation. Whereas this commutator is a bounded operator in the two first examples, it is a differential operator of order one (in the orbit directions) in the case of deformed Dirac operators, which makes the KK-product factorization of their index much less straightforward.  \\

One of the aims of this paper is to provide such a KK-product interpretation. Heuristically, the idea is rather natural, and involves viewing the perturbation as a dual-Dirac-like element $[\!\nu]$ in the orbit directions. This requires an `orbital Clifford algebra' $\Cl_{\Gamma}(M)$ introduced recently by Kasparov \cite{KasparovTransversallyElliptic}. A simple but key observation is that the operator $\st{D}$ determines a class in the $K$-homology of the crossed product algebra $G\ltimes \Cl_\Gamma(M)$ (see the Key Lemma in Section \ref{Section1}); the difficulty with the commutator mentioned above then disappears.\\

The rest of the paper explores some consequences by revisiting the index theorem, excision and cobordism invariance properties obtained by Braverman in \cite{Braverman2002}. From the perspective developed, the last two points become almost automatic and follow mostly from functorial arguments. Braverman's index theorem states that the analytic index of a deformed Dirac operator is equal to the topological index of a transversally elliptic symbol, obtained by deforming the symbol of $\st{D}$ by the vector field $\bm{\!\nu}$.  We show how this result can be deduced from the KK-product factorization and Kasparov's index theorem for transversally elliptic operators \cite[Theorem 8.18]{KasparovTransversallyElliptic}. This makes the relationship between the indices of deformed Dirac operators and of transversally elliptic operators more transparent.  It is also possible, as shown in \cite{MaZhang}, to relate such an index to an Atiyah-Patodi-Singer-type index, but this will not be discussed here. \\

A final note on quantization commutes with reduction in the case of a Hamiltonian $G$-space with proper moment map: many obvious similarities between the analytic approach relying on the properties of the deformed Dirac operator \cite{TianZhang, MaZhang, HochsSong3}, and the topological one based on $K$-theory classes of transversally elliptic symbols \cite{Paradan, ParadanVergne, ParadanFormal} may be spotted. In view of our last observation, it seems plausible that both approaches become essentially the same (up to Poincar\'e duality). It would be desirable to develop a synthesis of these methods in the framework of KK-theory, hopefully offering a unifying perspective on these works, and optimistically leading to conceptual simplifications. This is partly the motivation of the present paper, and will be the topic of future work. Some steps in this direction have been initiated in \cite{RodsphonQR}.\\

The contents of the paper are as follows:

$\bullet$ \emph{Section \ref{Section1}} reviews some material from \cite{KasparovTransversallyElliptic}, and in particular the notion of orbital Clifford algebra, which is used to build a transverse index class from the Dirac operator.

$\bullet$ \emph{Section \ref{Section2}} contains the main result, explaining how the equivariant index of deformed Dirac operators can be seen in terms of a KK-product. Preparatory material on deformed Dirac operators is included. 

$\bullet$ \emph{Section \ref{Section3}} revisits the excision and cobordism invariance of the index of deformed Dirac operators obtained in \cite{Braverman2002}, from the point of view developed in Section 2. 

$\bullet$ \emph{Section \ref{Section4}} reviews further material from \cite{KasparovTransversallyElliptic}, and derives Braverman's index theorem.  The level of knowledge of Kasparov theory expected of the reader is somewhat higher in this section.  Although the context and the $C^*$-algebras involved may be less familiar, the KK-theoretic techniques deployed are fairly standard.

$\bullet$ \emph{Appendices.} Certain arguments in the main body are streamlined if one has the flexibility to work on non-complete manifolds, and Appendix \ref{AppendixA} explains how to deal with this case. Evident extensions of classical results stated in Section 1 have their proofs relegated to Appendices \ref{AppendixB} and \ref{AppendixC}. \\

\textbf{Notation.} Throughout the article $G$ denotes a compact Lie group with Lie algebra $\g$.  We denote $\i = \sqrt{-1}$. Our convention for Clifford algebras is $\c(v)^2=-|v|^2$.  The notation $\la v\ra:=(1+|v|^2)^{1/2}$.  Given a Hermitian or Euclidean vector bundle $V \rightarrow M$ on a Riemannian manifold and section $s\colon M \rightarrow V$, we generally write $|s|$ for the point-wise norm of $s$, and $\|s\|$ for the $L^2$-norm using the Riemannian volume form.  If $A$ is a $C^\ast$-algebra and $\E$ a Hilbert $A$-module, we generally write $\mathcal{K}_A(\E)$ (resp. $\mathcal{B}_A(\E)$) for the compact (resp. adjointable) operators in the sense of Hilbert modules. Last but not least, we use the notation $\wh{\otimes}$ for graded tensor products. \\

\textbf{Acknowledgements.} We want to address very special thanks to G. Kasparov; how much the present article owes to his recent work (and the multiple discussions we had about it) will be evident thoughout the reading. Happy $70^\text{th}$ Birthday Genna! We also thank N. Higson and M. Braverman for helpful discussions.  Y. Song is supported by NSF grants DMS-1800667, 1952557.

\section{Transverse $K$-homology class of the Dirac operator} \setcounter{section}{1} \label{Section1}

Let $(M^n, g)$ be an even-dimensional Riemannian manifold (not necessarily complete) equipped with an isometric action of a compact\footnote{In fact, the results and proofs of this section remain valid even if $G$ is only a locally compact Lie group acting properly and isometrically on $M$.} Lie group $G$. Let $\g$ denote the Lie algebra of $G$.  Let $\Cliff(TM)$ denote the Clifford algebra bundle of $M$, and $\Cl_\tau(M)=C_0(M,\Cliff(TM))$ the $C^\ast$-algebra of continuous sections vanishing at infinity.

\subsection{Orbital Clifford algebra, and a key lemma}  We first review some material from the recent work of Kasparov \cite{KasparovTransversallyElliptic}.  For every $m \in M$, let   
\[ \rho_m \colon \beta \in \g \longmapsto \left. \dfrac{d}{dt}\right\vert_{t=0} e^{-t\beta} \cdot m \, \in T_m M, \]
denote the infinitesimal action at the point $m$.  We define
\[ \Gamma_m=\rho_m(\g)\subset T_mM \]
to be the tangent space to the orbit $G\cdot m$ at $m$. We would like to define spaces of `smooth' and `continuous' sections of $\sqcup_{m \in M} \Gamma_m$.  Since the orbits of a compact Lie group action typically vary in dimension (the map $m \mapsto \dim(\Gamma_m)$ is only lower semi-continuous in general), this takes a little care.\\  

Let $\rho \colon \g_M:=M\times \g \rightarrow TM$ denote the smooth bundle map induced by the maps $\rho_m$, i.e. $\rho$ is the \emph{anchor map} for the action Lie algebroid $\g_M$.  By post-composition $\rho$ induces a map (also denoted $\rho$) on sections.  We define the space of \emph{smooth and compactly supported sections of $\Gamma$} to be
\[ C^\infty_c(M,\Gamma):=\rho(C^\infty_c(M,\g_M)) \subset C^\infty_c(M,TM). \]
This is a simple instance of a \emph{singular foliation} in the sense of \cite{AndroulidakisSkandalis}: a $C^\infty_c(M)$-submodule of the space of smooth compactly supported vector fields which is involutive and locally finitely generated.  The space of continuous sections of $TM$ vanishing at infinity $C_0(M,TM)$ is the Banach space completion of $C^\infty_c(M,TM)$ with respect to the supremum norm.  We define the space of \emph{continuous sections of $\Gamma$ vanishing at infinity} $C_0(M,\Gamma)$ to be the closure of $C^\infty_c(M,\Gamma)$ in $C_0(M,TM)$.  Dropping the vanishing conditions we obtain similar definitions of the space of smooth sections and the space of continuous sections.  In particular, this endows $\Gamma=\sqcup_{m\in M}\Gamma_m$ with the structure of a continuous field of vector spaces over $M$, that we call the \emph{orbital tangent field}. (Recall that $\Gamma$ being a continuous field of vector spaces means that it admits a set of sections $\sigma$ that generate $\Gamma$ point-wise and such that $m\mapsto |\sigma_m|$ is a continuous function on $M$.)

\begin{definition} \label{def:CliffGamma}
The \emph{orbital Clifford algebra} $\Cl_\Gamma(M)$ is the $C^\ast$-subalgebra of $\Cl_\tau(M)$ generated by $C_0(M,\Gamma)$ and $C_0(M)$. $\Cl_\Gamma(M)$ is a $C_0(M)$-algebra, and may equivalently be described as the algebra of continuous sections vanishing at infinity of the continuous field of $C^\ast$-algebras $\Cliff(\Gamma)=\sqcup_{m \in M}\Cliff(\Gamma_m)$, where the continuous field structure is inherited from that of $\Gamma$.  $\Cl_\Gamma(M)$ contains a dense subalgebra generated by $C^\infty_c(M,\Gamma)$ and $C^\infty_c(M)$, denoted $\Cl^\infty_{\Gamma,c}(M)$. 
\end{definition}

Let $E$ be a $G$-equivariant $\bZ_2$-graded Hermitian Clifford module bundle over $M$, and let $\st{D}$ be a Dirac operator associated to a $G$-equivariant Clifford connection $\nabla$ on $E$. Locally, in terms of a local orthonormal frame $e_1,\dots e_n$, 
\[
\st{D} = \sum_{i=1}^n \c(e_i)\nabla_{e_i}. 
\]  
where $\c(\cdot)$ denotes Clifford multiplication.
 
\begin{lemma} \label{commutator lem}
Let $Y$ be a vector field on $M$ and $\nabla^{\mathrm{LC}}$ denote the Levi-Civita connection associated to the metric $g$.  Then, 
\[
[\c(Y), \st{D}] = -2\nabla_{Y} - \sum_{i=1}^n \c(e_i)\c(\nabla^{\mathrm{LC}}_{e_i}Y). 
\]
In particular, if $Y$ is a smooth section of the orbital tangent field $\Gamma$, then $[\c(Y), \st{D}]$ is a differential operator of order $1$ in the orbit direction.  
\end{lemma}


Let $\!\nu \colon M \rightarrow \g$ be a smooth map.  Introduce the differential operator $\L_{\!\nu}$ acting on smooth sections of $E$ by
\[ (\L_{\!\nu} \varphi)(m)=\left. \dfrac{d}{dt} \right\vert_{t=0} e^{t\!\nu(m)}\cdot \varphi(e^{-tv(m)}\cdot m).\]
Note that if $f \in C^\infty(M)$, then $\L_{f\!\nu}=f\L_{\!\nu}$.  We will use boldface $\bm{\!\nu} \colon m \in M \mapsto \rho_m(\!\nu(m))$ to denote the vector field generated by $\!\nu$.  The difference $\nabla_{\bm{\!\nu}}-\L_{\!\nu}$ is an operator of order $0$.  

\begin{definition}
\label{def:momentmap}
The \emph{moment map} for the pair $(E,\nabla)$ (cf. \cite{BerlineGetzlerVergne}) is the smooth section $\mu^E \in C^\infty(M,\g^\ast \wh{\otimes} \End(E))$ defined by the equation
\[ \pair{\mu^E}{\!\nu}=\nabla_{\bm{\!\nu}}-\L_{\!\nu} \]
for all $\!\nu \in C^\infty(M,\g)$.
\end{definition}
An element $h \in C^\infty(G)$ acts on $\varphi \in C_c^\infty(M,E)$ by the convolution operator
\[ (\C_h \varphi)(m)= h\star \varphi(m)=\int_G h(g)g\cdot \varphi(g^{-1}\cdot m)\, dg,\]
where $dg$ denotes a left invariant Haar measure.
\begin{keylemma} 
Let $\!\nu \colon M \rightarrow \g$ be a smooth map.  For any $h \in C^\infty(G)$ and $\varphi \in C_c^\infty(M,E)$,
\[ \L_{\!\nu} (h\star \varphi)(m)=-(\!\nu_m^R h)\star \varphi(m), \]
where $\!\nu_m^R$ denotes the right-invariant vector field on $G$ generated by $\nu_m:=\nu(m) \in \g$, which acts by differentiation on $h$.  For any $\chi \in C^\infty_c(M)$, the operator $\chi[\c(\bm{\!\nu} ),\st{D}]\C_h$ extends to a bounded operator on $L^2(M,E)$.  
\end{keylemma} 
\begin{proof}
By definition
\begin{align*}
\L_{\!\nu} (h\star \varphi)(m)
&=\left. \dfrac{d}{dt}\right\vert_{t=0}\int_G h(g)e^{t\!\nu_m}g\cdot \varphi(g^{-1}e^{-t\!\nu_m}\cdot m)\,dg\\
&=\left. \dfrac{d}{dt}\right\vert_{t=0}\int_G h(e^{-t\!\nu_m}g)g\cdot \varphi(g^{-1}\cdot m)\,dg\\
&=-(\!\nu_m^R h)\star \varphi(m)
\end{align*}
where in the second line we used a change of variables.\\

Lemma \ref{commutator lem} and Definition \ref{def:momentmap} show that the commutator $[\c(\bm{\!\nu}),\st{D}]=-2\L_{\!\nu}+B$ where $B$ is the bundle endomorphism
\[ B=-2\pair{\mu_E}{\nu}-\sum_{i=1}^n \c(e_i)\c(\nabla^{\mathrm{LC}}_{e_i}Y).\]
Let $\beta_1,...,\beta_{\dim(\g)}$ be a basis of $\g$, and let $\!\nu_j \colon M \rightarrow \bR$ be the components of $\!\nu$ relative to the basis.  By the calculation above
\[ \L_{\!\nu} (h\star \varphi)=\sum_j \!\nu_j (h_j\star \varphi) \]
where $h_j:=-\beta_j^R h \in C^\infty(G)$.  The second statement follows because the bundle endomorphism $B$ and the smooth functions $\!\nu_j$ are bounded on the support of $\chi$. 
\end{proof}

\subsection{K-homology and the transverse Dirac class} 
If $(M,g)$ is complete then $\st{D}$ is essentially self-adjoint, and the standard practice is to attach the following $K$-homology class to the Dirac operator: \ignore{\footnote{Strictly speaking a cycle for $K^0_G(A)$ is a $4$-tuple $(H,\pi_A,F,\pi_G)$ consisting of a Hilbert space $H$, a bounded operator $F$ on $H$, and representations $\pi_A$ (resp. $\pi_G$) of $A$ (resp. $G$) on $H$ satisfying several conditions.  Here and below we often omit the representations $\pi_A$, $\pi_G$ from the notation, leaving them implicit.}}
\begin{equation}  
\label{eqn:NormalClass}
[\st{D}_M]=\big[ (L^2(M,E), F = \st{D}(1+\st{D}^2)^{-\frac{1}{2}}) \big] \in K^0_G(C_0(M)). 
\end{equation}
Theorem \ref{thm:FundClass} below shows that the same pair $(L^2(M,E),F)$ defines a class in a different K-homology group.  This observation is due to Kasparov: see Lemma 8.8 of \cite{KasparovTransversallyElliptic} for the case of the de Rham-Dirac operator.\\

The representations of $G$ and $\Cl_\Gamma(M)$ on $L^2(M,E)$ by Clifford multiplication form a covariant pair, hence $L^2(M,E)$ carries a representation of the crossed-product algebra $G\ltimes \Cl_\Gamma(M)$.
\begin{theorem}  \label{thm:FundClass}
If $(M,g)$ is complete, the pair $(L^2(M,E), F)$ determines a class in $K^0(G\ltimes \Cl_\Gamma(M))$.
\end{theorem}

\begin{proof}
It suffices to verify that for every $a \in G \ltimes \Cl_\Gamma(M)$, $[F,a]$ is a compact operator. That the other Fredholm module axioms hold is analogous to standard cases.  We may assume $a = h\otimes \alpha$, with $h \in C^\infty(G)$ and $\alpha \in \Cl_{\Gamma,c}^\infty(M)$, since such elements are dense in $G\ltimes \Cl_\Gamma(M)$.  Let $\chi \in C^\infty_c(M)$ be a bump function equal to $1$ on the compact set $G\cdot \supp(\alpha)\subset M$.  Hence $a=\chi a$ and $[\st{D},a]=\chi [\st{D},a]$.  Following \cite{BaajJulg} or \cite[Lemma 4.2]{KasparovNovikov}, we first write $F$ as a Cauchy integral: 
\[ F = \frac{2}{\pi}\int_0^\infty \st{D} (1+\lambda^2 + \st{D}^2)^{-1} \, d\lambda,  \]
so that, 
\begin{equation} 
\label{eqn:comm}
[F,a] = \frac{2}{\pi}\int_0^\infty (1+\lambda^2 + \st{D}^2)^{-1}\big( (1+\lambda^2)\chi [\st{D},a] + \st{D}\chi [a,\st{D}]\st{D} \big)(1+\lambda^2 + \st{D}^2)^{-1} \, d\lambda. 
\end{equation}
Since $\st{D}$ is $G$-invariant, $[\st{D},a] = [\st{D},\alpha]\C_h$, and the latter is a bounded operator by virtue of the Key Lemma. The operator
\begin{equation} 
\label{eqn:normprod}
(1+\lambda^2)(1+\lambda^2+\st{D}^2)^{-1}\chi[\st{D},a]
\end{equation}
is compact by the Rellich lemma.  Since $\|(1+\lambda^2+\st{D}^2)^{-1}\|$ is $O(\lambda^{-2})$, the norm of \eqref{eqn:normprod} is uniformly bounded in $\lambda$, and the product
\[ (1+\lambda^2)(1+\lambda^2+\st{D}^2)^{-1}\chi[\st{D},a](1+\lambda^2 + \st{D}^2)^{-1} \]
is a compact operator with norm $O(\lambda^{-2})$.  For the second integrand
\begin{equation}  
\label{eqn:normprod2}
(1+\lambda^2 + \st{D}^2)^{-1} \st{D}\chi [a,\st{D}]\st{D}(1+\lambda^2 + \st{D}^2)^{-1} 
\end{equation}
note that $(1+\lambda^2+\st{D}^2)^{-1}\st{D}\chi$ is compact by the Rellich lemma, and has norm $O(\lambda^{-1})$.  Using again the fact that $[a,\st{D}]$ is bounded, it follows that \eqref{eqn:normprod2} is a compact operator of norm $O(\lambda^{-2})$.  Thus both integrands in \eqref{eqn:comm} are compact with norm $O(\lambda^{-2})$, hence the integral converges in the norm topology to a compact operator.
\end{proof}

If $M$ is not complete then \cite[Chapter 10]{HigsonRoe} explains a slightly more elaborate construction that produces a class $[\st{D}_M] \in K^0_G(C_0(M))$ from a Dirac operator. (One could also replace the metric on $M$ with one that is complete, and this leads to the same K-homology class.)  Using similar techniques it is not difficult to do the same in our setting, and we outline how this is done in Appendix \ref{AppendixA}.  Granted this we make the following definition.
\begin{definition}
Let $M$ be a Riemannian manifold with an isometric action of a compact Lie group $G$.  Let $\st{D}$ be a $G$-equivariant Dirac operator acting on sections of a Clifford module bundle $E$.  The Hilbert space $L^2(M,E)$ and the operator $\st{D}$ determine a K-homology class $[\st{D}_{M,\Gamma}]\in K^0(G\ltimes \Cl_\Gamma(M))$ that we refer to as the \emph{transverse Dirac class} associated to $\st{D}$.  If $M$ is complete, this is the class described in Theorem \ref{thm:FundClass}.  For the general case see Appendix \ref{AppendixA}.
\end{definition}

\begin{remark}\label{rem:AgreeOnSingDist}
Two well-known facts about $[\st{D}_M] \in K^0_G(C_0(M))$ are that (i) the class does not depend on the metric, and (ii) in the complete case the operator $F$ can be replaced by $\chi(\st{D})$ where $\chi$ is any `normalizing function', cf. \cite[Chapter 10]{HigsonRoe}.  Similar results hold for $[\st{D}_{M,\Gamma}] \in K^0(G\ltimes \Cl_\Gamma(M))$.  Given two $G$-invariant complete metrics $g_0$, $g_1$ on $M$, there is a canonical isometric isomorphism $(\Gamma, g_0) \rightarrow (\Gamma, g_1)$ given fiberwise by the square-root of the composite map
\[ \Gamma_m \xlongrightarrow{g_0^\flat} \Gamma_m^\ast \xlongrightarrow{g_1^\sharp} \Gamma_m. \qquad  \text{(Flat and sharp exponents denote metric contractions.)}  \]
This induces a canonical isomorphism $\Cl_{\Gamma}(M, g_0) \rightarrow \Cl_\Gamma(M, g_1)$ between the corresponding orbital Clifford algebras, and so also between the crossed products by $G$.  These isomorphisms intertwine the corresponding classes in $K^0(G\ltimes \Cl_\Gamma(M,g_i))$, and in this sense $[\st{D}_{M,\Gamma}]$ is independent of the metric. 
\end{remark}

\subsection{Significance of the class $[\st{D}_{M,\Gamma}]$} \label{subsection:Significance}
We discuss here briefly why the class $[\st{D}_{M,\Gamma}] \in \K^0(G\ltimes \Cl_\Gamma(M))$ is referred to as a transverse index class, by explaining its relationship with the more familiar class $[\st{D}_M] \in \K^0_G(C_0(M))$. The result stated is included for expository purposes, without proof and at the cost of some rigor. However it suggests an interesting geometric interpretation of the class $[\st{D}_{M,\Gamma}]$ in the spirit of non-commutative geometry and index theory of foliations, and might provide some helpful insights to the reader. \\

\ignore{Let $\pi \colon G\ltimes \Cl_\Gamma(M) \rightarrow \mathcal{B}(L^2(M,E))$ be the representation of $G\ltimes \Cl_\Gamma(M)$ on $L^2(M,E)$ mentioned in Theorem \ref{thm:FundClass}.  The image $\pi(G\ltimes \Cl_\Gamma(M))$ is a $C^\ast$-subalgebra of $\mathcal{B}(L^2(M,E))$, and the class $[\st{D}_{M,\Gamma}]$ may be viewed as a class in $\KK(\pi(G\ltimes \Cl_\Gamma(M)),\bC)$.  Moreover the latter group can be viewed as a direct summand in the group $\KK^G(\pi(G\ltimes \Cl_\Gamma(M)),\bC)$ by \cite[Lemma 8.1]{KasparovTransversallyElliptic} (the $G$-action on $G\ltimes \Cl_\Gamma(M)$ combines the $G$-action on $\Cl_\Gamma(M)$ and left translation on $G$).\\}

Let $(\beta_1, \ldots, \beta_{\dim(\g)})$ be a basis of $\g$. Following Kasparov \cite[Definition 8.3]{KasparovTransversallyElliptic} we define the \emph{orbital Dirac operator} by 
\[D_{\Gamma} : G \ltimes \Cl_\Gamma^\infty(M) \to G \ltimes \Cl_\Gamma^\infty(M) \quad ; \qquad D_{\Gamma} = \sum_{j=1}^{\dim(\g)} \c(\rho(\beta_j)) \L_{\beta_j} \] 
where $\Cl_\Gamma^\infty(M)$ is the smooth version of $\Cl_\Gamma(M)$, and $\L_{\beta_j}$ denotes Lie differentiation for the diagonal action of $G$ on $G\ltimes \Cl_\Gamma^\infty(M)$ given by $(g\odot a)(h)=g\cdot a(g^{-1}h)$.  This definition leads (after considerable work\footnote{The receptacle of the class $[D_\Gamma]$ given here is technically not right, and more sophisticated KK-groups have to be used. However, it is sufficient for the purpose of exposition and motivation, especially since it will not be used thereafter.}, compare \cite[Definition 8.5]{KasparovTransversallyElliptic}) to the construction of an element
\[ [D_\Gamma] \in \KK^G(C_0(M), G \ltimes \Cl_{\Gamma}(M)). \] 
If all the orbits of $G$ have the same dimension, this element is the longitudinal index class of a family of Dirac operators over the orbit space $M/G$. The following theorem extends this observation, and shows that the class $[\st{D}_{M,\Gamma}] \in \KK(G \ltimes \Cl_\Gamma(M),\bC)$ previously constructed should be interpreted as a \emph{transverse index class}.

\begin{theorem}  \label{thm:Significance} $[\st{D}_M]=[\st{D}_\Gamma]\wh{\otimes}_{G\ltimes \Cl_\Gamma(M)}[\st{D}_{M,\Gamma}]$.
\end{theorem}

\emph{Comment on the proof.} Kasparov proves this theorem in \cite[Theorem 8.9]{KasparovTransversallyElliptic} in the case where the Clifford module is the exterior algebra bundle $\Lambda^{\bullet} T^*M \wh{\otimes} \bC$, with $\st{D}$ being the de Rham-Dirac operator $d+d^{*}$. The resulting class $[d_{M,\Gamma}] \in K^0(G \ltimes \Cl_{\tau \oplus \Gamma}(M))$ lies in a slightly different KK-group, but is in essence the same as the one from Theorem \ref{thm:FundClass} (This class is used later in Section \ref{Section4}). The theorem above can be proved by a straightforward readaptation of Kasparov's arguments, or by a direct reduction to the special case he deals with. $\hfill{\square}$ 

\subsection{Restriction to open sets}\label{sec:restrict}
Let $U$ be a $G$-invariant open set of $M$, let $\iota_U \colon C_0(U)\hookrightarrow C_0(M)$ be the extension-by-$0$ homomorphism, and $\iota_U^\ast \colon \K^0_G(C_0(M))\rightarrow \K^0_G(C_0(U))$ the corresponding restriction map on K-homology.  A well-known property of the class $[\st{D}_M] \in \K^0_G(C_0(M))$ (cf. \cite[Proposition 10.8.8]{HigsonRoe}) is that
\[ \iota_U^\ast[\st{D}_M]=[\st{D}_U] \]
where $[\st{D}_U]\in \K^0_G(C_0(U))$ is the class determined by the restriction $\st{D}|_U$.\\

The class $[\st{D}_{M,\Gamma}]$ has an analogous property. We will abuse notation slightly and use $\iota_U$ to also denote the extension-by-$0$ homomorphism $\Cl_\Gamma(U)\hookrightarrow \Cl_\Gamma(M)$, as well as the induced $^\ast$-homomorphism between the crossed products $G\ltimes \Cl_\Gamma(U)\hookrightarrow G\ltimes \Cl_\Gamma(M)$. Thus there is a restriction map
\[ \iota_U^\ast \colon \K^{0}(G\ltimes \Cl_\Gamma(M)) \rightarrow \K^{0}(G\ltimes \Cl_\Gamma(U)).\]

\begin{proposition} \label{prop:OpenRestrict} The restriction of $\st{D}$ to $U$ determines a class $[\st{D}_{U,\Gamma}] \in \K^0(G\ltimes \Cl_\Gamma(U))$ and $\iota_U^\ast[\st{D}_{M,\Gamma}]=[\st{D}_{U,\Gamma}]$.
\end{proposition}

For a proof, see Appendix \ref{AppendixB}. 

\subsection{Manifolds with boundary.}\label{sec:boundary}
Let $\ti{M}$ be a Riemannian $G$-manifold with boundary, and let $M=\partial \ti{M}$ be the boundary, equipped with the restriction of the metric and of the $G$-action. There is a short exact sequence of $C^\ast$-algebras
\[ 0 \rightarrow C_0(\ti{M}\setminus M)\rightarrow C_0(\ti{M})\rightarrow C_0(M) \rightarrow 0\]
which induces a corresponding $6$-term exact sequence in K-homology. Let
\[ \partial \colon \K^1(C_0(\ti{M}\setminus M))\rightarrow \K^0(C_0(M)) \]
be the induced boundary homomorphism.\\

Suppose $\ti{E} \rightarrow \ti{M}$ is an ungraded Clifford module bundle on the odd-dimensional manifold $\ti{M}$.  A Dirac operator $\ti{\st{D}}$ for $\ti{E}|_{\ti{M}\setminus M}$ determines a class $[\ti{\st{D}}_{\ti{M}\setminus M}]\in \K^1(C_0(\ti{M}\setminus M))$. Let $E=\ti{E}|_{\partial \ti{M}}$, equipped with $\bZ_2$-grading $E^{\pm}$ the $\pm \i$-eigenbundles of $\c(n)$, where $n$ is an inward unit normal vector to the boundary. A well-known property of the class $[\ti{\st{D}}_{\ti{M}\setminus M}]$ (cf. \cite[Proposition 11.2.15]{HigsonRoe}) is that
\[ \partial[\ti{\st{D}}_{\ti{M}\setminus M}]=[\st{D}_M] \]
where $[\st{D}_M]\in \K^0(C_0(M))$ is the class associated to a Dirac operator acting on sections of $E$.\\

Transverse Dirac classes have an analogous property. The definitions of $\Gamma$ and of the orbital Clifford algebra $\Cl_\Gamma(\ti{M})$ go through for the manifold with boundary $\ti{M}$.  Moreover the definition of $\Gamma$ is compatible with restriction to the boundary, in the sense that the restriction of $\Gamma$ to the boundary (in the sense of continuous fields), coincides with the orbital tangent field of the boundary.  We therefore make a slight abuse of notation and write $\Gamma$ for the orbital tangent fields on each of $\ti{M}$, $\ti{M}\setminus M$ and $M$.\\ 

There is a surjective $\ast$-homomorphism $\Cl_\Gamma(\ti{M}) \rightarrow \Cl_\Gamma(M)$ given by restriction.  Since the boundary is $G$-invariant, there is an extension of $C^\ast$-algebras
\begin{equation*}
\label{eqn:BdShortExact}
0\rightarrow G\ltimes \Cl_\Gamma(\ti{M}\setminus M) \rightarrow G\ltimes \Cl_\Gamma(\widetilde{M}) \rightarrow G\ltimes \Cl_\Gamma(M) \rightarrow 0,
\end{equation*}  
and a corresponding boundary map in K-homology:
\[ \partial \colon \K^{1}(G\ltimes \Cl_\Gamma(\ti{M}\setminus M)) \rightarrow \K^{0}(G\ltimes \Cl_\Gamma(M))\]

It is straight-forward to adapt the arguments in Theorem \ref{thm:FundClass} and Appendix \ref{AppendixA} to show that a Dirac operator $\ti{\st{D}}$ acting on sections of $\ti{E}|_{\ti{M}\setminus M}$ yields a class $[\ti{\st{D}}_{\ti{M}\setminus M,\Gamma}] \in \K^1(G\ltimes \Cl_\Gamma(\ti{M}\setminus M))$.

\begin{proposition} \label{prop:DiracBoundary} 
Let $\ti{E}$ be an (ungraded) Clifford module over the odd-dimensional manifold $\ti{M}$, and $[\ti{\st{D}}_{\ti{M}\setminus M,\Gamma}] \in K^1(G\ltimes \Cl_\Gamma(\ti{M}\setminus M))$ the corresponding class.  Let $E=\ti{E}|_{\partial \ti{M}}$, equipped with $\bZ_2$-grading $E^{\pm}$ the $\pm \i$-eigenbundles of $\c(n)$, where $n$ is an inward unit normal vector to the boundary.  Then  
\[ \partial[\ti{\st{D}}_{\ti{M}\setminus M,\Gamma}]=[\st{D}_{M,\Gamma}]. \]
\end{proposition}

For a proof, see Appendix \ref{AppendixC}.

\section{Deformed Dirac operator and KK-product} \label{Section2}

In this section we assume $(M,g)$ is a complete Riemannian $G$-manifold (without boundary).

\subsection{Deformed Dirac operator}

Let us first review some definitions introduced by Braverman \cite{Braverman2002}. A \emph{taming map} is a $G$-equivariant map
\[ \!\nu \colon M \rightarrow \g \]
such that the induced vector field $\bm{\!\nu} : m \in M \mapsto \rho_m(\!\nu(m))$ has a compact vanishing locus. It is convenient to assume that $|\bm{\!\nu}|\le 1$ with equality outside a compact neighborhood of the vanishing locus (one can always achieve this after re-scaling $\!\nu$ by a suitable smooth positive function). Following Braverman \cite{Braverman2002}, a non-negative $G$-invariant function $f \in C^{\infty}(M)^G$ is said to be \emph{admissible} if
\[
\lim_{M \ni m \to \infty} \frac{f^2}{|df|_M + f(|\nabla^{\tn{LC}}\bm{\!\nu}|_M+|\!\nu|_{\g}+|\pair{\mu^E}{\!\nu}|_E) + 1} = \infty. 
\]
(In this expression, $|\cdot|_M$ is used to denote the point-wise norms on the vector bundles $TM\simeq T^\ast M$ and $\End(TM)$ induced by the Riemannian metric, $|\cdot|_E$ denotes the point-wise norm on the vector bundle $\End(E)$ induced by the Hermitian structure, and $|\cdot|_\g$ denotes the norm on the Lie algebra $\g$ induced from its inner product.) One can show \cite[Lemma 2.7]{Braverman2002} that admissible functions always exist.

\begin{definition} Let $E \to M$ be a Clifford module bundle and let $\st{D}$ be a Dirac operator acting on sections of $E$. Let $\!\nu \colon M \rightarrow \g$ be a taming map and let $f$ be an admissible function. The \emph{deformed Dirac operator} is the Dirac-type operator
\[
\st{D}_{f\!\nu} =\st{D} + \i f\c(\bm{\!\nu}).  
\]
\end{definition}

Intuitively, the assumption that $f$ be admissible ensures that the cross-terms in $\st{D}_{f\!\nu}^2$ can be neglected. This is reminiscent of Kasparov's technical theorem, which provides operators playing the same role in the general construction of the KK-product. The admissibility property also ensures nice properties of the spectrum of $\st{D}_{f\!\nu}^2$ (cf. proof of Lemma \ref{lem:CptResolvent}), which makes it possible to define an equivariant index.  

\begin{theorem}[Braverman \cite{Braverman2002}] \label{thm Braverman}
Let $\st{D}_{f\!\nu}$ be a deformed Dirac operator associated to a $\bZ_2$-graded Clifford module bundle $E=E^+ \oplus E^-$.  Then the pair $(L^2(M,E),\st{D}_{f\!\nu})$ determines a class  
\[ [\st{D}_{f\!\nu}] \in \KK_0(C^\ast(G),\bC), \]
which is independent of the choice of admissible function $f$.  Under the identification $\KK_0(C^\ast(G),\bC)\simeq R^{-\infty}(G)=\bZ^{\wh{G}}$ given by the Peter-Weyl theorem, the class $[\st{D}_{f\!\nu}]$ identifies with its index
\[
\mathrm{Ind}(\st{D}_{f\!\nu}) := \sum_{\pi\in \wh{G}} (m^+_\pi - m^-_\pi) \cdot \pi \, \in R^{-\infty}(G)
\]
where $m^{\pm}_\pi < \infty$ is the multiplicity of the irreducible representation $\pi \in \wh{G}$ in $\ker(\st{D}_{f\!\nu})\cap L^2(M,E^{\pm})$. 
\end{theorem}
To give some idea of what is involved, we outline an argument. Let $F_{f\!\nu}=\st{D}_{f\!\nu}(1+\st{D}_{f\!\nu}^2)^{-\frac{1}{2}}$. Recall that $\C_h$ denotes the operator of convolution by $h$.  First, for every $h \in C^*(G)$, $[F_{f\!\nu}, \C_h]=0$ by $G$-invariance of $F_{f\!\nu}$. It only remains to see that $(1-F_{f\!\nu}^2)\C_h$ is compact, which comes from the following lemma:

\begin{lemma} \label{lem:CptResolvent}
Let $h \in C^*(G)$. Then, $(1+\st{D}_{f\!\nu}^2)^{-1} \C_h$ is a compact operator on $L^2(M,E)$. 
\end{lemma}

\begin{proof} \emph{(of the lemma)} 
For $G$ compact, the Peter-Weyl theorem states that $C^\ast(G)$ is an infinite direct sum over matrix algebras $\End(\pi)$, $\pi \in \wh{G}$.  It suffices to consider the case where $h$ lies in a single summand $\End(\pi)$ (in other words, $h$ is a matrix coefficient for $\pi$). Equivalently we must show that the restriction of $(1+\st{D}_{f\!\nu}^2)^{-1}$ to each isotypic component in $L^2(M,E)$ is compact.  One has
\[ \st{D}_{f\!\nu}^2 = \st{D}^2 + f^2 | \bm{\!\nu} |^2 + \i \, \big( f [\st{D}, \c(\bm{\!\nu})] + \c(df)\c(\bm{\!\nu}) \big) \]
In terms of a local orthonormal frame $e_1,...,e_{\dim(M)}$ the commutator writes
\[ [\st{D},\c(\bm{\! \!\nu})]=\L_{\!\nu}+\sum_j \c(e_j)\c(\nabla^{\tn{LC}}_{e_j}\bm{\! \!\nu})+\pair{\mu^E}{\!\nu}.\]
On the $\pi$-isotypic component, one has an inequality of semi-bounded operators $|\L_{\!\nu}|\le C_\pi |\!\nu|$ (the latter is a multiplication operator for the function $|\!\nu|$ on $M$) with $C_\pi$ a constant just depending on the representation $\pi$.  Thus on the $\pi$-isotypic component one has an inequality of semi-bounded operators
\begin{equation} 
\label{eqn:ineqsemi}
\st{D}_{f\!\nu}^2 \ge \st{D}^2 + f^2\Big( | \bm{\!\nu} |_M^2 - f^{-2}\big( f(C_\pi|\!\nu|_\g+|\nabla^{\tn{LC}}\bm{\! \!\nu}|_M+|\pair{\mu^E}{\!\nu}|_E)  + |df|_M|\bm{\!\nu}|_M \big)\Big).
\end{equation}
The definition of admissible function implies that the term in the inner brackets, multiplied by the factor of $f^{-2}$, goes to $0$ at infinity.  On the other hand $|\bm{\!\nu}|_M^2=1$ outside a compact set in $M$. Consequently on the $\pi$-isotypic component, there is an inequality of semi-bounded operators of the form
\[ \st{D}_{f\!\nu}^2 \ge \st{D}^2+V \]
where the potential function $V$ is proper and bounded below.  It is known that the operator $\st{D}^2+V$  has discrete spectrum (cf. \cite[Appendix B]{LSQuantLG} for a short proof and further references).  This implies $\st{D}_{f\!\nu}$ restricted to the $\pi$-isotypical component has discrete spectrum, and hence compact resolvent.
\end{proof}

\vspace*{0.1cm}

\subsection{KK-product factorization}
We now come to the main result of the article, which is a KK-product factorization of the $K$-homology class $[\st{D}_{f\!\nu}] \in \KK_0(C^*(G), \bC)$. \\

Given two $C^\ast$-algebras $A$ and $B$, we denote $\mathbb{E}(A,B)$ the set of $(A,B)$ KK-cycles (or Kasparov $A,B$-bimodules). Recall the following theorem, which allows to recognize when a KK-cycle arises as a KK-product.  

\begin{theorem} [Connes-Skandalis, \cite{Skandalis}]
\label{KK def}
Let $A, B, C$ be graded $C^\ast$-algebras, with $A$ separable. Let 
\[
(H_1, \pi_1, F_1) \in \mathbb{E}(A, B), \hspace{5mm} (H_2, \pi_2, F_2) \in \mathbb{E}(B, C),
\]
and let $\E_1$, $\E_2$ be their respective KK-theory classes. Suppose that $F \in \mathcal{L}_C(H_1 \wh{\otimes}_B H_2)$ is a $C$-linear bounded operator such that
\begin{enumerate}
\item[\emph{(a)}] $ (H = H_1 \wh{\otimes}_B H_2, \pi_1 \wh{\otimes} 1, F) \in \mathbb{E}(A,C), $ 

\item[\emph{(b)}] $F$ is an $F_2$-connection, i.e for every $\xi \in H_1$, the operators $(\xi \wh{\otimes} \, \bm{.}) F_2 - (-1)^{\tn{deg}(\xi)}F (\xi \wh{\otimes} \, \bm{.})$ and $(\xi \wh{\otimes} \, \bm{.})^* F - (-1)^{\tn{deg}(\xi)}F_2 (\xi \wh{\otimes} \, \bm{.})^{*}$ are compact operators. 

\item[\emph{(c)}] For every $a \in A$, $a [ F_1 \wh{\otimes}_B 1, F ] a^* \geq 0$ modulo compact operators on $H$.
\end{enumerate}

Then, the cycle $(H, \pi_1 \wh{\otimes} 1, F) \in \mathbb{E}(A,C)$ represents the KK-product $\E_1 \wh{\otimes}_B \E_2 \in \KK(A,C)$. Moreover, the $\KK$-product $\E_1 \wh{\otimes}_B \E_2$ always admits a representative of this form, which is unique up to (norm-continuous) homotopy.
\end{theorem}

Now, consider a deformed Dirac operator $\st{D}_{f\!\nu} =\st{D} + \i f \c(\bm{\!\nu})$, where $f$ is an admissible function and $\bm{\!\nu}$ is the vector field associated to the taming map $\!\nu : M \to \g$, with $|\bm{\!\nu}| = 1$ outside a compact neighborhood of the zero set of $\bm{\!\nu}$. The latter condition means that the vector field $\bm{\!\nu}$ determines a class 
\[ [\!\nu]=\big[\big(\Cl_\Gamma(M), \i\c(\bm{\!\nu}) \big)\big] \in \KK^G_0(\bC,\Cl_\Gamma(M)). \]
Let 
\[ j^G[\!\nu] \in \KK_0(C^\ast(G),G\ltimes \Cl_\Gamma(M)) \]
be its image under the descent map $j^G \colon \KK_0^G(\bC, \Cl_\Gamma(M)) \to \KK_0(C^\ast(G),G\ltimes \Cl_\Gamma(M))$.  We can then form the product
\[ j^G[\!\nu]\wh{\otimes}_{G\ltimes \Cl_\Gamma(M)} [\st{D}_{M,\Gamma}] \in \KK_{0}(C^\ast(G),\bC).\]

\begin{theorem} \label{thm:KK-product}
The $K$-homology class $[\st{D}_{f\!\nu}] \in \KK_{0}(C^\ast(G),\bC)$ of the deformed Dirac operator factors as the following KK-product:
\[
 [\st{D}_{f\!\nu}] = j^G[\!\nu]\wh{\otimes}_{G\ltimes \Cl_\Gamma(M)} [\st{D}_{M,\Gamma}] \in \KK_{0}(C^\ast(G),\bC). 
\]
\end{theorem}

\begin{proof}
The first condition of the Connes-Skandalis criterion (Theorem \ref{KK def}) is Theorem \ref{thm Braverman}. It suffices to check the $F$-connection condition for $\xi=a \in G\ltimes \Cl_\Gamma(M)$ of the form $a=h\wh{\otimes} \alpha$ with $h \in C^\infty(G)$, $\alpha \in \Cl^\infty_c(M)$. The operator denoted $(a\wh{\otimes} \, \bm{.}) \colon L^2(M,E)\rightarrow (G\ltimes \Cl_\Gamma(M))\wh{\otimes}_{G\ltimes \Cl_\Gamma(M)} L^2(M,E)\simeq L^2(M,E)$ in the Connes-Skandalis criterion is given by the action of $a \in G\ltimes \Cl_\Gamma(M)$ on $L^2(M,E)$, hence we must verify that $F_{f\!\nu}a-(-1)^{\tn{deg}(a)}aF$ is a compact operator on $L^2(M,E)$. Let $\chi \in C^\infty_c(M)$ be a bump function equal to $1$ on the compact set $G\cdot \supp(\alpha) \subset M$. Let
\[ B=\st{D}_{f\!\nu}a-(-1)^{\tn{deg}(a)}a\st{D}=[\st{D},a]+\i f\c(\bm{\!\nu})a.\]
Then $B=\chi B$ and it follows from the Key Lemma that $B$ is a bounded operator.  Using integral expressions as in the proof of Theorem \ref{thm:FundClass}, one has
\[ F_{f\!\nu}a-(-1)^{\tn{deg}(a)}aF=\frac{2}{\pi}\int_0^\infty (1+\lambda^2+\st{D}_{f\!\nu}^2)^{-1}\big((1+\lambda^2)\chi B-(-1)^{\tn{deg}(a)}\st{D}_{f\!\nu}\chi B\st{D}\big)(1+\lambda^2+\st{D}^2)^{-1}d\lambda.\]
As in the proof of Theorem \ref{thm:FundClass}, the integrand is compact with operator norm $O(\lambda^{-2})$, hence the integral converges in norm to a compact operator.  The verification for $(a\wh{\otimes} \, \bm{.})^\ast$ is similar.\\

We now check the positivity condition.  Recall that for $G$ compact $C^\ast(G)$ is isomorphic to the direct sum over $\pi \in \wh{G}$ of matrix algebras $\End(\pi)$.  It suffices to consider $h \in C^\ast(G)$ lying in a single summand $\End(\pi)$.  Recall $\C_h$ denotes the operator of convolution by $h$.  Write the commutator $[\i\c(\bm{\!\nu}), F_{f\!\nu}]$ via an integral formula for $F_{f\!\nu}$ as in the proof of Theorem \ref{thm:FundClass}:
\begin{multline}
\label{eqn:posit} 
\C_h [\i\c(\bm{\!\nu}), F_{f\!\nu}] \C_h^{\ast} = \frac{2}{\pi}\int_0^\infty (1+\lambda^2 + \st{D}_{f\!\nu}^2)^{-1} \Big((1+ \lambda^2) \, \C_h [\i\c(\bm{\!\nu}), \st{D}_{f\!\nu}]  \C_h^{\ast} \\
+ \st{D}_{f\!\nu} \, \C_h [\i\c(\bm{\!\nu}), \st{D}_{f\!\nu}] \C_h^{\ast}\, \st{D}_{f\!\nu} \Big)(1+\lambda^2 + \st{D}_{f\!\nu}^2)^{-1}d\lambda.
\end{multline}
The integral formula for $F_{f\!\nu}$ is convergent in the strong operator topology.  Here, we have used the $G$-equivariance of $\c(\bm{\!\nu})$ and $\st{D}_{f\!\nu}$, which implies that they commute with $\C_h$. Consider the graded commutator
\[ [\i\c(\bm{\!\nu}),\st{D}_{f\!\nu}]=\i[\c(\bm{\!\nu}),\st{D}]+f|\bm{\!\nu}|^2. \]
It follows from the admissibility condition on $f$ and our assumption that $|\bm{\!\nu}|=1$ outside a compact set that the function
\[ f\big(|\bm{\!\nu}|^2-f^{-1}(C_\pi |\!\nu|+|\nabla^{\tn{LC}}\bm{\!\nu}|+|\pair{\mu^E}{\!\nu}|)\big)\]
is bounded below, where $C_\pi$ is the constant appearing in inequality \eqref{eqn:ineqsemi}; let $-\infty<b \le 0$ be any (strictly) lower bound. It follows from the proof of Lemma \ref{lem:CptResolvent} (see especially inequality \eqref{eqn:ineqsemi}) that the operator
\[ P=[\i\c(\bm{\!\nu}),\st{D}_{f\!\nu}]-b\]
is a positive unbounded operator when restricted to the $\pi$-isotypical component of $L^2(M,E)$.  Thus
\[ \C_h[\i\c(\bm{\!\nu}),\st{D}_{f\!\nu}]\C_h^\ast=\C_hP\C_h^\ast+b\C_h\C_h^\ast, \]
and $\C_hP\C_h^\ast$ is a positive operator. The contribution of $P$ to the integrand in \eqref{eqn:posit} is a positive operator, and the corresponding integral converges in the strong operator topology to a positive operator.\\ 

The contribution to the integral \eqref{eqn:posit} of $b\C_h\C_h^\ast$ is 
\begin{equation}
\label{eqn:posit2} 
\frac{2b}{\pi}\int_0^\infty (1+\lambda^2 + \st{D}_{f\!\nu}^2)^{-1} \Big((1+ \lambda^2)\C_h\C_h^{\ast}
+ \st{D}_{f\!\nu}\C_h\C_h^{\ast} \st{D}_{f\!\nu} \Big)(1+\lambda^2 + \st{D}_{f\!\nu}^2)^{-1}d\lambda.
\end{equation}
The two terms in the integrand are analysed as in the proof of Theorem \ref{thm:FundClass}.  For example consider
\begin{equation} 
\label{eqn:Cterm}
b(1+\lambda^2 + \st{D}_{f\!\nu}^2)^{-1} \st{D}_{f\!\nu} \C_h\C_h^{\ast}\st{D}_{f\!\nu}(1+\lambda^2 + \st{D}_{f\!\nu}^2)^{-1}.
\end{equation}
By Lemma \ref{lem:CptResolvent} the operator $(1+\lambda^2+\st{D}_{f\!\nu}^2)^{-1}\st{D}_{f\!\nu}\C_h$ is compact, with norm $O(\lambda^{-1})$, and the same is true of its adjoint.  Thus \eqref{eqn:Cterm} is a compact operator with norm $O(\lambda^{-2})$.  It follows that the integral \eqref{eqn:posit2} converges in norm to a compact operator.
\end{proof}

\begin{remark} 
In the case when $M$ is compact, the equivariant index of $\st{D}$ can be obtained by applying the collapse map $M \to \text{pt}$ to the class $[\st{D}] \in K^0_G(C_0(M))$. In the present non-compact situation, the result above shows that the map $(j^G[\!\nu] \wh{\otimes} \, \bm{.})$ plays a similar role.   
\end{remark} 

\section{Applications} \label{Section3}

In this section let $M$ be a complete Riemannian manifold equipped with an isometric action of a compact Lie group $G$, and let $\st{D}_{f\!\nu} = \st{D} + \i f\c(\bm{\!\nu})$ be a deformed Dirac operator associated to a ($\bZ_2$-graded) Clifford module bundle $E \to M$. 

\subsection{Excision for deformed Dirac operators.}\label{sec:excision}
A first consequence of the KK-product factorization of $\st{D}_{f\!\nu}$ is an excision result for its index, which can be seen as a rough $K$-theoretic analogue of localization formulas in equivariant cohomology.\\ 

Recall that we assumed $|\bm{\!\nu}|=1$ outside a compact set.  Let $U \subset M$ be a $G$-invariant open set such that $|\bm{\!\nu}|=1$ outside $U$, and let $\iota_U \colon \Cl_\Gamma(U)\hookrightarrow \Cl_\Gamma(M)$ be the extension-by-$0$ homomorphism.  Let $\!\nu_U=\!\nu|_U$.  The pair $(\Cl_\Gamma(U),\c(\bm{\!\nu}_U))$ determines a class $[\!\nu_U]\in \KK^G_0(\bC,\Cl_\Gamma(U))$.
\begin{proposition} \label{prop:excision} $(\iota_U)_\ast[\!\nu_U]=[\!\nu] \in \KK^G(\bC,\Cl_\Gamma(M))$.\ignore{

Let $U \subset M$ be a $G$-invariant open set such that $|\bm{\!\nu}|=1$ outside $U$. Then
\[ \big[ \big(\Cl_\Gamma(U),\c(\bm{\!\nu}\vert_{U}) \big) \big] = \big[ \big(\Cl_\Gamma(M),\c(\bm{\!\nu}) \big) \big] \in \KK^G_0(\bC,\Cl_\Gamma(M)). \]}
\end{proposition}
\begin{proof}
Under the obvious identification $\Cl_\Gamma(U)\wh{\otimes}_{\Cl_\Gamma(U)}\Cl_\Gamma(M)\simeq \Cl_\Gamma(U)$, the element $(\iota_U)_\ast[(\Cl_\Gamma(U),\c(\bm{\!\nu}_U))]=[(\Cl_\Gamma(U),\c(\bm{\!\nu}_U))]\wh{\otimes}_{\Cl_\Gamma(U)} [\iota_U]$ is represented by the pair $[(\Cl_\Gamma(U),\c(\bm{\!\nu}_U))] \in \KK^G(\bC,\Cl_\Gamma(M))$. Then, a homotopy between this cycle and the cycle $(\Cl_\Gamma(M),\c(\bm{\!\nu}))$ is provided by the following $(\bC,\Cl_\Gamma(M)\wh{\otimes} C[0,1])$-cycle $(\E,\F)$:
\[ \E=\{\text{continuous functions }f \colon [0,1] \rightarrow \Cl_\Gamma(M) \, : \, \supp(f(1)) \subset U \};  \quad \F=\i \c(\bm{\!\nu}) .\]
That $1-\F^2=1-|\bm{\!\nu}|^2$ is a compact operator on $\E$ comes from the fact that $|\bm{\!\nu} |^2 = 1$ outside of $U$, whence the result. 
\end{proof}

\begin{corollary}
$[\st{D}_{f\!\nu}]=j^G[\!\nu_U]\wh{\otimes}_{G\ltimes \Cl_\Gamma(U)}[\st{D}_{U,\Gamma}]$.
\end{corollary}
\begin{proof}
This follows from the KK-product factorization of Theorem \ref{thm:KK-product}, Proposition \ref{prop:excision}, plus associativity of the Kasparov product:
\[ [\st{D}_{f\!\nu}]=(j^G[\!\nu_U]\wh{\otimes} [\iota_U])\wh{\otimes} [\st{D}_{M,\Gamma}]=j^G[\!\nu_U]\wh{\otimes}([\iota_U]\wh{\otimes} [\st{D}_{M,\Gamma}]) \]
together with the fact that $[\iota_U]\wh{\otimes}_{G\ltimes \Cl_\Gamma(M)}[\st{D}_{M,\Gamma}]=[\st{D}_{U,\Gamma}]$ (Proposition \ref{prop:OpenRestrict}). 
\end{proof}

The corollary together with another application of Theorem \ref{thm:KK-product} on the manifold $U$, imply that the index of $\st{D}_{f\!\nu}$ can be computed from the index of a deformed Dirac operator on $U$.  This operator is determined up to suitable homotopy by the condition that it represents the KK-product $j^G[\!\nu_U]\wh{\otimes}_{G\ltimes \Cl_\Gamma(U)}[\st{D}_{U,\Gamma}]$. Note however that one cannot simply restrict $\st{D}_{f\!\nu}$ to $U$; one should for example complete the metric on $U$ and also replace $f|_U$ with a function that is admissible for $U$. This result was proved by Braverman \cite{Braverman2002} using the cobordism invariance of the index (see the next section).  Here we obtain it as a consequence of the KK-product factorization.

\subsection{Cobordism invariance of the index} \label{sec:cobinv}
We will reprove the following result of Braverman \cite{Braverman2002}, which leads directly to the cobordism invariance of the index of the deformed Dirac operator. 
\begin{theorem}
\label{thm:CobInv}
Let $M$ be a Riemannian $G$-manifold which is the boundary of a Riemannian $G$-manifold $\ti{M}$.  Let $\ti{E}$ be a $G$-equivariant (ungraded) Clifford module bundle over $\ti{M}$, and let $E=\ti{E}|_M$ be the induced Clifford module bundle over the boundary $M$ with $\bZ_2$-graded subbundles $E^{\pm}$ given by the $\pm \i$-eigenbundles of $\c(n)$, where $n$ is the inward unit normal vector to the boundary. Let $\st{D}$ be a Dirac operator associated to $E$, let $\ti{\!\nu}\colon \ti{M} \rightarrow \g$ be a taming map and let $\!\nu$ be its restriction to $M$. Then
\[ j^G[\!\nu] \wh{\otimes}_{G\ltimes \Cl_\Gamma(M)} [\st{D}_{M,\Gamma}]=0 \in K^0(C^*(G)).\]
\end{theorem}


\subsubsection{Review of cobordism invariance in the standard case.}
Let us first recall the Baum-Douglas Taylor proof of cobordism invariance in the standard case (cf. \cite[p.765]{BDTRelativeCycles}), i.e we assume $\ti{M}$ (and then $M$) is compact, and ignore the $G$-action. The key $C^\ast$-algebra extension is
\begin{equation} 
\label{eqn:ShortExact1}
0 \rightarrow C_0(\ti{M}\setminus M) \rightarrow C_0(\ti{M}) \xrightarrow{r} C_0(M) \rightarrow 0
\end{equation}
where $r$ denotes restriction to the boundary.  The proof of cobordism invariance is based on the analogue of Proposition \ref{prop:DiracBoundary}:
\[ \partial [\ti{\st{D}}]=[\st{D}] \]
where $[\ti{\st{D}}]\in K^1(C_0(\ti{M}\setminus M))$ is the K-homology class defined by the Dirac operator $\ti{\st{D}}$ on the odd-dimensional (open) manifold $\ti{M} \setminus M$, and $\partial$ is the boundary homomorphism in the six term exact sequence (in K-homology) associated to \eqref{eqn:ShortExact1}.\\ 

Let $\ti{p}$ (resp. $p$) denote the homomorphism $\bC \rightarrow C_0(\ti{M})$ (resp. $\bC \rightarrow C_0(M)$) obtained from the collapsing map $\ti{M}\rightarrow \pt$ (resp. $M \rightarrow \pt$).  Hence
\begin{equation} 
\label{eqn:RestrictToBd}
r\circ \ti{p}=p \quad \Rightarrow \quad \ti{p}^\ast \circ r^\ast=p^\ast.
\end{equation}  

We have
\[ p^\ast [\st{D}]=\ti{p}^\ast \circ r^\ast \circ \partial [\ti{\st{D}}]\]
but the middle composition $r^\ast \circ \partial=0$ because it is the composition of two successive maps in the six term sequence. $\hfill{\square}$

\subsubsection{Proof of Theorem \ref{thm:CobInv}.}
The relevant $C^\ast$-algebra extension to consider in this case is
\begin{equation} 
\label{eqn:ShortExact2}
0 \rightarrow G \ltimes \Cl_\Gamma(\ti{M}\setminus M) \rightarrow G \ltimes  \Cl_\Gamma(\ti{M}) \xrightarrow{r} G \ltimes \Cl_\Gamma(M) \rightarrow 0.
\end{equation}
where $r$ is also the restriction map. The replacements for the collapsing maps $\ti{p}$, $p$ are the taming maps $\ti{\!\nu}$, $\!\nu$ which define elements $j^G[\ti{\!\nu}] \in \KK_0(C^\ast(G), G\ltimes \Cl_\Gamma(\ti{M}))$, $j^G[\!\nu] \in \KK_0(C^\ast(G),G\ltimes \Cl_\Gamma(M))$ respectively.  Then, we have
\[ j^G[\!\nu]=j^G[\ti{\!\nu}]\wh{\otimes} [r^\ast] \]
which is the analogue of equation \eqref{eqn:RestrictToBd} (we regard the $\ast$-homomorphism $r$ as an element $[r^\ast] \in \KK(G\ltimes \Cl_\Gamma(\ti{M}), G\ltimes \Cl_\Gamma(M))$ here).  Thus
\[ j^G[\!\nu]\wh{\otimes} [\st{D}_{M,\Gamma}]=j^G[\ti{\!\nu}]\wh{\otimes} [r^\ast] \wh{\otimes} [\ti{\st{D}}_{M,\Gamma}]=j^G[\ti{\!\nu}]\wh{\otimes} [r^\ast] \wh{\otimes} \partial [\ti{\st{D}}_{\ti{M},\Gamma}], \]
where the second equality uses Proposition \ref{prop:DiracBoundary}. But
\[ [r^\ast]\wh{\otimes} \partial[\ti{\st{D}}_{\ti{M},\Gamma}]=r^\ast \circ \partial[\ti{\st{D}}_{\ti{M},\Gamma}]\]
and $r^\ast \circ \partial=0$ for the same reason as before: it is the composition of two successive maps in the six term exact sequence for \eqref{eqn:ShortExact2}. This completes the proof. $\hfill{\square}$

\section{Deformed Dirac operators and transversally elliptic operators}\setcounter{section}{4} \label{Section4}
In this section, we provide a KK-theoretic proof of the following theorem due to Braverman \cite[Theorem 5.5]{Braverman2002} (see also \cite{MaZhangBrav,LSWit}).

\begin{theorem} \label{thm:BravInd} Let $M$ be a complete Riemannian $G$-manifold equipped with an isometric action of a compact Lie group $G$, and let $\st{D}_{f\!\nu}$ be a deformed Dirac operator. Then, the equivariant index of $\st{D}_{f\!\nu}$ in $R^{-\infty}(G)$ is equal to the index (in Atiyah's sense) of the transversally elliptic symbol $\sigma_{\!\nu}^0(\xi)=\i\c(\xi+\bm{\!\nu})$ obtained by deforming the symbol of the Dirac operator using the vector field $\bm{\!\nu}$.
\end{theorem}

Such transversally elliptic deformations have interesting applications; we mention for example the work of Paradan \cite{Paradan} on the quantization-commutes-with-reduction theorem in symplectic geometry.  \\

The idea of the proof is relatively simple: we observe that with the appropriate KK-groups, the product of $[\!\nu]$ and of an appropriate symbol class $[\sigma_{M,\Gamma}]$ of $\st{D}$ is the K-theory class of the transversally elliptic symbol $\sigma_{\!\nu}^0$ defined above. The result then follows from our KK-product factorization and a KK-theoretic Poincar\'e duality theorem for transversally elliptic operators obtained by Kasparov \cite[Theorem 8.18]{KasparovTransversallyElliptic}. \\

The first four subsections of this section might be viewed as a brief further `invitation' to Kasparov's work \cite{KasparovTransversallyElliptic}; we do not attempt to be exhaustive, but rather describe a small sample of the many new constructions and results contained in \emph{loc. cit.}, in view of deriving Theorem \ref{thm:BravInd}.

\subsection{The transverse de Rham and Dolbeault classes.}
Recall that there is a canonical K-homology class $[d_M]\in \KK^G(\Cl_\tau(M),\bC)$ associated to the de Rham-Dirac operator acting on differential forms. Denoting $\Cl_{\tau\oplus\Gamma}(M) = \Cl_{\tau}(M) \wh{\otimes}_{C_0(M)} \Cl_{\Gamma}(M)$, a similar construction to \ref{thm:FundClass} applied to the de Rham-Dirac operator on $L^2(M,\wedge T^\ast M)$ produces a class $[d_{M,\Gamma}]\in \KK(G\ltimes \Cl_{\tau\oplus\Gamma}(M),\bC)$ that we refer to as the \emph{transverse de Rham class} (cf. \cite[Definition-Lemma 8.8]{KasparovTransversallyElliptic}).\\ 

\begin{definition}
\label{def:cliffsymb}
Let $(E,\c \colon \Cliff(TM)\rightarrow \End(E))$ be a $G$-equivariant Clifford module bundle on $M$. We define a Hilbert $\Cl_\tau(M)$-module $\E_\tau$ as follows: the underlying Banach space is $C_0(M,E)$, the (right) $\Cl_\tau(M)$-module structure is determined on generators $\xi \in C_0(M,TM)$ by the formula
\[ e\cdot \xi=\i (-1)^{\deg(e)}\c(\xi)e, \]
and the $\Cl_\tau(M)$-valued inner product is given point-wise as the composition
\[ E\otimes E \rightarrow E^* \otimes E=\End(E)\simeq \End_{\Cl}(E)\otimes \Cliff(TM) \xrightarrow{2^{-n/2}\tn{tr}\otimes \id} \Cliff(TM), \]
where the first map uses the isomorphism $E \simeq E^*$ determined by the Hermitian structure and $\End_{\Cl}(E)$ denotes endomorphisms of $E$ that commute with the $\Cliff(TM)$ action.
\end{definition}
The class $[\st{D}_M]\in \KK^G(C_0(M),\bC)$ associated to the Dirac operator on $E$ factors as a KK-product (cf. \cite[Definition 3.9, Proposition 3.10]{KasparovTransversallyElliptic})
\begin{equation} 
\label{e:DiracE}
[\st{D}_M]=[\E_\tau]\wh{\otimes}_{\Cl_\tau(M)}[d_M] 
\end{equation}
where $[\E_\tau]\in \R\KK^G(M;C_0(M),\Cl_\tau(M))$\footnote{If $A$ and $B$ be $C_0(M)$-$C^*$-algebras, recall that the bivariant K-group $\R\KK^G(M;A,B)$ is defined the same way as $\KK(A,B)$, with the following additional requirement: if $(H,F)$ is a KK-cycle, then for every $f \in C_0(M), a \in A, b \in B, \xi \in H$, one has $(fa)\xi b = a \xi (fb)$} is the class represented by the cycle having Hilbert $\Cl_\tau(M)$-module $\E_\tau=C_0(M,E)$ and the zero operator.  One has a similar result for the classes $[\st{D}_{M,\Gamma}]$, $[d_{M,\Gamma}]$.  To state it, recall that there is a product in $\R\KK$ (cf. \cite[Proposition 2.21]{KasparovNovikov}):
\[ \wh{\otimes}_{M}\colon \R\KK^G(M;A,B)\times \R\KK^G(M;C,D)\rightarrow \R\KK^G(M;A\wh{\otimes}_{C_0(M)}C,B\wh{\otimes}_{C_0(M)}D).\]
We will write $1_{\Cl_\Gamma(M)} \in \R\KK^G(M;\Cl_\Gamma(M),\Cl_\Gamma(M))$ for the class represented by the pair $(\Cl_\Gamma(M),0)$.  The following is the natural analogue of \eqref{e:DiracE}, and can be checked without difficulty, using for instance Theorem \ref{KK def}.
\begin{proposition}
\label{prop:factorCliffDirac}
There is a factorization
\[ [\st{D}_{M,\Gamma}]=j^G([\E_\tau]\wh{\otimes}_{M}1_{\Cl_\Gamma(M)})\wh{\otimes}_{G\ltimes \Cl_{\tau\oplus\Gamma}(M)}[d_{M,\Gamma}] \in \KK(G \ltimes \Cl_\Gamma(M), \bC). \]
\end{proposition}

There is a well-known class $[d_\xi]\in \R\KK^G(M;C_0(TM),\Cl_\tau(M))$ implementing the KK-equivalence between the algebras $C_0(TM)$ and $\Cl_\tau(M)$ (cf. \cite[Definition 2.5]{KasparovTransversallyElliptic}); it can be described explicitly via a family of Dirac operators $\mathcal{D}_m$, $m \in M$ on the family of Hilbert spaces $L^2(T_mM,\Cliff(T_mM))$, $m \in M$ for the fibres of the bundle $\pi_{TM}\colon TM \rightarrow M$.
\begin{definition}[\cite{KasparovTransversallyElliptic}, Definition 8.17]
\label{def:transdol}
Let $\Cl_\Gamma(TM):=C_0(TM)\otimes_{C_0(M)}\Cl_\Gamma(M)$ (beware this is not exactly the orbital Clifford algebra of the $G$-manifold $TM$).  The \emph{transverse Dolbeault class} is the product
\begin{equation}
\label{eqn:DolDeRham}
[\ol{\partial}^{\cl}_{TM,\Gamma}]=j^G([d_\xi]\wh{\otimes}_{M}1_{\Cl_\Gamma(M)})\wh{\otimes}_{G\ltimes \Cl_{\tau\oplus\Gamma}(M)}[d_{M,\Gamma}] \in \KK(G\ltimes \Cl_\Gamma(TM),\bC).
\end{equation}
\end{definition}

The symbol $\sigma(\xi)=\i \c(\la \xi\ra^{-1}\xi)$ of the bounded transform $F=\st{D}(1+\st{D}^2)^{-1/2}$ of the Dirac operator determines a class $[\sigma_M]=[(C_0(TM,\pi_{TM}^\ast E),\sigma(\xi))]\in \R\KK^G(M;C_0(M),C_0(TM))$.  By \cite[Proposition 3.10]{KasparovTransversallyElliptic}, 
\begin{equation}
\label{eqn:CliffSymb}
[\sigma_M]\wh{\otimes}_{C_0(TM)}[d_\xi]=[\E_\tau] \in \R\KK^G(M;C_0(M),\Cl_\tau(M)).
\end{equation}
In Kasparov's terminology, the element in \eqref{eqn:CliffSymb} is referred to as the \emph{Clifford symbol} of $\st{D}$. Proposition \ref{prop:factorCliffDirac} and equations \eqref{eqn:DolDeRham}, \eqref{eqn:CliffSymb} give us the formula
\begin{equation}
\label{eqn:DiracDol}
[\st{D}_{M,\Gamma}]=j^G([\sigma_M]\wh{\otimes}_{M}1_{\Cl_\Gamma(M)})\wh{\otimes}_{G\ltimes \Cl_\Gamma(TM)}[\ol{\partial}^{\cl}_{TM,\Gamma}] \in \KK(G\ltimes \Cl_\Gamma(M),\bC).
\end{equation}

\subsection{Transversally elliptic symbols and the symbol algebra $\S_\Gamma(M)$.}\label{sec:symbalg}
For the purpose of motivation, suppose $M$ is a \emph{compact} Riemannian manifold (we will drop the compactness assumption shortly).  Let $A$ be a $G$-equivariant pseudo-differential operator with symbol $\sigma_A$ acting on sections of a $G$-equivariant Hermitian vector bundle $E$.\footnote{By the `symbol' $\sigma_A$ of $A$, we will mean a section of $\pi_{TM}^\ast \End(E)$ in the usual H{\"o}rmander $(\rho=1,\delta=0)$ class, defined everywhere and not required to be homogeneous, whose equivalence class modulo symbols of lower order is the class of the principal symbol of $A$.}  The \emph{support} $\supp(\sigma_A)$ of $\sigma_A$ is the subset of $T^\ast M\simeq TM$ where $\sigma_A$ fails to be invertible.  The operator $A$ is said to be \emph{transversally elliptic} if $\supp(\sigma_A)\cap T_GM$ is compact, where $T_GM\simeq T_G^\ast M=\tn{ann}(\Gamma)$ is the conormal space to the $G$-orbits.  In this case Atiyah proved \cite{Atiyah} that the restriction $A_\pi$ ($\pi \in \wh{G}$) of $A$ to each isotypical component is Fredholm, hence $A$ has a well-defined `index',\footnote{Atiyah proved a stronger result, that the index determines a distribution on $G$.} 
\begin{equation} 
\label{e:AtiyahDef}
\index(A)=\sum_{\pi \in \wh{G}} \index(A_\pi)\pi \in R^{-\infty}(G)=\bZ^{\wh{G}}.
\end{equation}
Moreover, the index depends only on the class in $\K^0_G(T_GM)=\K^G_0(C_0(T_GM))$ defined by the symbol.\\

\ignore{In KK-theory, cycles for $\K^0_G(T_GM)=\KK^G(\bC,C_0(T_GM))=\R\KK^G(M;C(M),C_0(T_GM))$ are pairs $(\E,\sigma)$, where $\E=C_0(T_GM,\pi_{T_GM}^\ast E)$ is the space of continuous sections vanishing at infinity of the pullback of a $G$-equivariant $\bZ_2$-graded Hermitian vector bundle $E$ on $M$, and $\sigma \in C_0(T_GM,\pi_{T_GM}^\ast \End(E))$ is an odd self-adjoint bundle endomorphism (not necessarily the symbol of a pseudodifferential operator) with the crucial property $(1-\sigma^2)\in C_0(T_GM)$.}

 \ignore{Roughly speaking, one needs to work with bounded bundle endomorphisms $\sigma$ over $TM$ that have a little more in common with (normalized) symbols of transversally elliptic order-$0$ pseudodifferential operators.} 
 
However the K-theory group of the algebra $C_0(T_GM)$ turns out to not be ideal for the purpose of stating an index theorem.  Kasparov's replacement for $C_0(T_GM)$ in this context is the following.\ignore{ improvement of Berline-Vergne's notion of transversally good symbols \cite{BerlineVergne}.}

\begin{definition}[\cite{KasparovTransversallyElliptic}, Definition-Lemma 6.2] \label{def:symbalg}
Let $M$ be a Riemannian manifold (not necessarily compact) with an isometric action of a compact Lie group $G$.  The \emph{symbol algebra} $\S_\Gamma(M)$ is the norm-closure in $C_b(TM)$ (the algebra of continuous bounded functions on $TM$) of the set of all smooth, bounded functions $b(m,\xi)$ on $TM$, which are compactly supported in the $m$ variable, and satisfy the following two conditions:
\begin{enumerate}
\item The exterior derivative $d_mb(m,\xi)$ in $m$ is norm-bounded uniformly in $\xi$, and there is an estimate $|d_\xi b(m,\xi)|\le C(1+|\xi|)^{-1}$ for a constant $C$ which depends only on $b$ and not on $(m,\xi)$.
\item The restriction of $b$ to $T_GM$ belongs to $C_0(T_GM)$.
\end{enumerate}
Given a $G$-equivariant $\bZ_2$-graded Hermitian vector bundle $E$, we can similarly define a Hilbert $\S_\Gamma(M)$-module, denoted $\S_\Gamma(E)$, as the norm-closure in the space of bounded sections of the pull-back bundle $\pi_{TM}^\ast E$ satisfying similar conditions to those in Definition \ref{def:symbalg} (using the norm on the fibres of $\pi_{TM}^\ast E$ induced by the Hermitian structure). 
\end{definition}

We now return to our usual setting, with $M$ a complete Riemannian $G$-manifold.  From now on, we refer to transversally elliptic operators (or symbols) according to the following definition.

\begin{definition} 
\label{def:transell}
Let $A$ be a properly supported, odd, self-adjoint $G$-invariant pseudodifferential operator of order 0 acting on sections of a $G$-equivariant $\bZ_2$-graded Hermitian vector bundle $E$. We will say that $A$ (or its symbol $\sigma_A$) is transversally elliptic if for every $a \in C_0(M)$,  $a\cdot(1-\sigma_A^2)\in \S_\Gamma(M)$. 
\end{definition}

\ignore{We emphasize once again that the definition is stronger than the standard notion of transverse ellipticity (unless the group is trivial).} Since $\S_\Gamma(M) \subset \mathcal{K}_{\S_\Gamma(M)}(\S_\Gamma(E))$ (the compact operators on $\S_\Gamma(E)$ in the Hilbert module sense), a transversally elliptic symbol determines a class
\[ [\sigma_A]=[(\S_\Gamma(E),\sigma_A)]\in \R\KK^G(M;C_0(M),\S_\Gamma(M)).\]
By construction there is a $\ast$-homomorphism $\iota_{T_GM}^\ast \colon \S_\Gamma(M)\rightarrow C_0(T_GM)$, hence a map
\[ \R\KK^G(M;C_0(M),\S_\Gamma(M))\rightarrow \R\KK^G(M;C_0(M),C_0(T_GM)).\]
In this sense the element $[\sigma_A]\in \R\KK^G(M;C_0(M),\S_\Gamma(M))$ can be viewed as a `refinement' of the `naive' class in $\R\KK^G(M;C_0(M),C_0(T_GM))$ defined by the symbol. 

\subsection{The class $\f_{M,\Gamma}$.}
Recall the trivial bundle $\g_M=M \times \g$ and the anchor map $\rho \colon \g_M \rightarrow TM$ describing the vector fields generated by the $G$-action.  We now fix a $G$-invariant metric $(-,-)_{\g_M}$ on the bundle $\g_M$ such that $g(\rho(\beta),\rho(\beta))\le (\beta,\beta)_{\g_M}$.  Using the metrics on $\g_M$, $TM$ the anchor $\rho$ has a transpose
\[ \rho^\top \colon TM \rightarrow \g_M.\]
\begin{definition}
\label{def:varphi}
We define a smooth bundle map $\varphi \colon TM \rightarrow TM$ to be the composition $\varphi=\rho\circ \rho^\top$.
\end{definition}
\begin{remark}
The range of $\varphi$ is contained in $\Gamma \subset TM$, and $\varphi$ is, roughly speaking, a smooth version of fibre-wise orthogonal projection $T_mM \rightarrow \Gamma_m$.  For simplicity suppose the metric on $\g_M$ is constant.  Let $\beta^1,...,\beta^{\dim(\g)}$ be an orthonormal basis of $\g$, and $\beta^1_M=\rho(\beta^1),...,\beta^{\dim(\g)}_M=\rho(\beta^{\dim(\g)})$ the corresponding vector fields on $M$.  Let $X$ be a vector field.  Then
\[ \varphi(X)=\sum_{j=1}^{\dim(\g)}g(X,\beta^j_M)\beta^j_M.\]
If the action of $G$ is free, then $\varphi(X)$ is, to a first approximation, the projection of $X$ to the orbit directions (with some re-scaling of its components).  At the other extreme, in a neighborhood of an isolated fixed point, the length $|\beta^j_M|$ is $O(r)$ where $r$ is the distance to the fixed-point, and consequently the length $|\varphi(X)|$ is $O(r^2)$ (the typical example would be the vector field $r\partial_\theta$ in $\bR^2$).
\end{remark}

Having defined the map $\varphi$, we may state another useful description of the symbol algebra $\frak{S}_\Gamma(M)$, which interprets its elements as symbols having negative order in the transverse directions:

\begin{lemma}[\cite{KasparovTransversallyElliptic}, Definition-Lemma 6.2]
\label{lem:symbolest}
Under item \emph{(a)} in the definition of $\frak{S}_\Gamma(M)$ above, item \emph{(b)} is equivalent to the following estimate: for any $\varepsilon >0$ there exists a constant $c_\varepsilon>0$ such that
\[ |b(m,\xi)|\le c_\varepsilon \frac{\la \varphi_m(\xi)\ra^2}{\la \xi \ra^2}+\varepsilon, \qquad \forall m \in M, \xi \in T_mM.\]
\end{lemma}

The following definition is one of the main reasons to introduce the symbol algebra $\frak{S}_{\Gamma}(M)$.

\begin{definition} \label{def:fMgam} \cite[pp.1344--1345]{KasparovTransversallyElliptic}
The element $[\f_{M,\Gamma}]\in \R\KK^G(M;\S_\Gamma(M),\Cl_\Gamma(TM))$ is the class represented by the pair $(\Cl_\Gamma(TM),\f_{M,\Gamma})$ where at a point $(m,\xi) \in T_mM$, the operator $\f_{M,\Gamma}(m,\xi)$ is left Clifford multiplication by $-\i \varphi_m(\xi)\la \varphi_m(\xi)\ra^{-1}$. (Recall from Definition \ref{def:transdol} that $\Cl_\Gamma(TM):=C_0(TM)\otimes_{C_0(M)}\Cl_\Gamma(M)$. For $\varphi \in C^\infty(M,\End(TM))$ see Definition \ref{def:varphi}.)
\end{definition}
Note that for $b \in \S_\Gamma(M)$, the estimate in Lemma \ref{lem:symbolest} shows that the product $b(m,\xi)(1-\f_{M,\Gamma}(m,\xi)^2)=b(m,\xi)\la \varphi_m(\xi)\ra^{-2}$ belongs to $C_0(TM)\subset \Cl_\Gamma(TM)=\mathcal{K}_{\Cl_\Gamma(TM)}(\Cl_\Gamma(TM))$. Hence the pair $(\Cl_\Gamma(TM),\f_{M,\Gamma})$ does define a KK-cycle.

\begin{remark} The class $[\f_{M,\Gamma}]$ should be viewed as the symbol class of the orbital Dirac element sketched in Section \ref{subsection:Significance}. On the other hand, it implements a KK-equivalence between $\S_\Gamma(M)$ and $\Cl_\Gamma(TM)$.
\end{remark}

\subsection{Kasparov's index theorem for transversally elliptic operators.}
Let $X$ be a compact Riemannian manifold equipped with an isometric action of the compact Lie group $G$.  Let $A$ be a $G$-equivariant, odd, self-adjoint order-$0$ pseudodifferential operator acting on sections of a $\bZ_2$-graded Hermitian vector bundle $E$.  Suppose the symbol $\sigma_A$ is transversally elliptic in the sense of Definition \ref{def:transell}. Then
\begin{itemize}
\item The symbol determines a class $[\sigma_A]\in \R\KK(X;C_0(X),\S_\Gamma(X))$.
\item The pair $(L^2(X,E),A)$ determines a class $[A] \in \KK(G\ltimes C_0(X),\bC)$, and moreover $\index(A) \in R^{-\infty}(G)\simeq \K^0(C^\ast(G))$ is the push-forward of $[A]$ under the map $p\colon X \rightarrow \pt$ (\cite{JulgTransEll}, \cite[Proposition 6.4]{KasparovTransversallyElliptic}). Indeed to demonstrate the latter point, recall $C^*(G)=\oplus_{\pi \in \wh{G}} \End(V_\pi)$
by the Peter-Weyl theorem. The projection $e_\pi \in C^*(G)$ corresponding to the summand $\End(V_\pi)$ determines a K-theory class $[e_\pi]\in K_0(C^*(G))$, and by definition the index pairing $\pair{[e_\pi]}{p_\ast [A]}=\index(e_\pi Ae_\pi)=\index(A_\pi)$, compatible with \eqref{e:AtiyahDef}.
\end{itemize}
Kasparov's index theorem relates these two KK-theory classes.  To state it, it is convenient to introduce a variant of the symbol class.
\begin{definition}[\cite{KasparovTransversallyElliptic}, Definition 8.13]
\label{def:tcl}
Recall the element $[\f_{X,\Gamma}]\in \R\KK^G(X;\S_\Gamma(X),\Cl_\Gamma(TX))$ introduced in Definition \ref{def:fMgam}. The \emph{tangent Clifford symbol class} $[\sigma_A^{\tn{tcl}}]$ is the $\KK$-product
\[ [\sigma_A^{\tn{tcl}}]=[\sigma_A]\wh{\otimes}_{\S_\Gamma(X)}[\f_{X,\Gamma}] \in \R\KK^G(X;C_0(X),\Cl_\Gamma(TX)).\]
In the sequel it will be convenient to use similar notation in a slightly broader context: if $[\omega] \in \KK^G(\scr{A},\S_\Gamma(X))$ (or $\R\KK$) for some $C^*$-algebra $\scr{A}$, then we will write $[\omega^{\tn{tcl}}]$ as shorthand for the product $[\omega]\wh{\otimes}_{\S_\Gamma(X)}[\f_{X,\Gamma}] \in \KK^G(\scr{A},\Cl_\Gamma(TX))$.
\end{definition}
Kasparov provides (see the paragraph following \cite[Definition 8.13]{KasparovTransversallyElliptic}) the following explicit cycle $(\E_{TX,\Gamma},S_{TX,\Gamma})$ representing the class $[\sigma_A^{\tn{tcl}}]$: the Hilbert module is the tensor product
\[ \E_{TX,\Gamma}=C_0(TX,\pi_{TX}^\ast E)\wh{\otimes}_{C_0(TX)}\Cl_\Gamma(TX) \]
and the operator $S_{TX,\Gamma}$ is
\begin{equation} 
\label{eqn:tcl}
N_1^{1/2}(\sigma_A\wh{\otimes} 1)+N_2^{1/2}(1\wh{\otimes} \f_{X,\Gamma})
\end{equation}
where the weights $N_1,N_2=1-N_1 \in C^\infty_b(TX)$ take the form
\begin{equation}
\label{eqn:weights}
N_1(x,\xi)=\frac{\la\xi\ra^2}{\la\xi\ra^2+\la \varphi_x(\xi)\ra^3},\qquad \qquad N_2(x,\xi)=\frac{\la \varphi_x(\xi)\ra^3}{\la\xi\ra^2+\la \varphi_x(\xi)\ra^3}.
\end{equation}
For later use, we note that the weights $N_1$, $N_2$ are chosen according to Kasparov's technical theorem, and have the following important properties: 
\begin{lemma}\label{lem:decayN1}
$N_1^{1/2}\cdot \S_\Gamma(X)\subset C_0(TX)$ and $N_2(1-\f_{X,\Gamma}^2) \in C_0(TX)$.
\end{lemma}
In fact, the first inclusion holds even when $X$ is non-compact (this will be used later).  
\ignore{
\begin{proof}
The property of $N_2$ is immediate from
\[ N_2(x,\xi)(1-\f_{X,\Gamma}(x,\xi)^2)=\frac{\la\varphi_x(\xi)\ra}{\la\xi\ra^2+\la\varphi_x(\xi)\ra^3}\le \frac{\la\varphi_x(\xi)\ra}{\la\xi\ra^2}\le \frac{1}{\la\xi\ra}.\] 
We turn to the property of $N_1$, and we no longer assume $X$ is compact.  It suffices to check $N_1^{1/2}\cdot b \in C_0(TX)$ for $b$ in a dense subset of $\S_\Gamma(X)$, so in particular we may assume $\supp(b)\subset \pi_{TX}^{-1}(K)$ for a compact $K\subset X$.  Let $f\in C_0(X)$, $f\ge 0$ be equal to $1$ on $K$, so $b=f\cdot b$.  Applying Lemma \ref{lem:symbolest}, we get that for every $\varepsilon>0$ there is a constant $c_\varepsilon$ such that
\begin{align*} 
N_1^{1/2}(x,\xi)|b(x,\xi)|&\le c_\varepsilon f(x)\frac{\la\varphi_x(\xi)\ra^2}{\la\xi\ra(\la\xi\ra^2+\la \varphi(\xi)\ra^3)^{1/2}}+f(x)N_1^{1/2}(x,\xi)\varepsilon \\ 
&\le c_\varepsilon f(x)\frac{\la\varphi_x(\xi)\ra^2}{\la\xi\ra \la \varphi_x(\xi)\ra^{3/2}}+f(x)N_1^{1/2}(x,\xi)\varepsilon\\
&\le c_\varepsilon \frac{f(x)}{\la\xi\ra^{1/2}}+f(x)N_1^{1/2}(x,\xi)\varepsilon.
\end{align*}
In the last expression, the first term is in $C_0(TX)$ while the second term has norm $\le \|f\|_\infty \varepsilon$.  Since we get an estimate like this for any $\varepsilon>0$, it follows that $N_1^{1/2}b \in C_0(TX)$.
\end{proof}
}
The lemma shows that $N_1\big(\mathcal{K}_{\S_\Gamma(X)}(\S_\Gamma(E))\wh{\otimes}1\big) \subset \mathcal{K}_{\Cl_\Gamma(TX)}\big(\S_\Gamma(E)\wh{\otimes}_{\S_\Gamma(X)}\Cl_\Gamma(TX)\big)$ and $N_2(1-\f_{X,\Gamma}^2)\in \mathcal{K}_{\Cl_\Gamma(TX)}\big(\S_\Gamma(E)\wh{\otimes}_{\S_\Gamma(X)}\Cl_\Gamma(TX)\big)$, as in the general construction of the KK-product (cf. the proof of \cite[Theorem 18.4.3]{Blackadar}).  It follows from standard arguments that $(\E_{TX,\Gamma},S_{TX,\Gamma})$ is indeed a cycle representing $[\sigma_A^{\tn{tcl}}]$. It is not hard and rather instructive to check this fact together with the previous lemma by hand. \\ 

With these preparations, we can finally state Kasparov's index theorem for transversally elliptic operators.
\begin{theorem}[\cite{KasparovTransversallyElliptic}, Theorem 8.18]\label{thm:KaspInd}
Let $A$ be a transversally elliptic operator on a compact Riemannian $G$-manifold $X$.  Then,
\ignore{Let $X$ be a compact Riemannian $G$-manifold, and let $A$ be a transversally elliptic operator\ignore{ satisfying the conditions stated at the beginning of this section}.  Then}
\[ [A]=j^G([\sigma_A^{\tn{tcl}}])\wh{\otimes}_{G\ltimes \Cl_\Gamma(TX)}[\ol{\partial}^{\cl}_{TX,\Gamma}] \in \KK(G\ltimes C_0(X),\bC).\]
\end{theorem}
\begin{remark}
Kasparov gives several other variants of the index theorem, but this version is best suited to our purposes. Moreover, his theorem still applies if $X$ and $G$ are non-compact, as long as $G$ acts properly and isometrically on $X$. We will only need the compact case.
\end{remark}

\ignore{\begin{remark}
One motivation for calling this an `index theorem' is by analogy with the elliptic case, where Kasparov explains that the pairing of the ordinary symbol $[\sigma_A]$ with the Dolbeault-Dirac operator $[\ol{\partial}_{TX}]$ coincides with the topological index of $[\sigma_A]$ in the sense of Atiyah and Singer, cf. \cite[Remarks 4.4, 4.5]{KasparovTransversallyElliptic}.  
\end{remark}}

\subsection{Transversally elliptic symbols on open manifolds.}\label{sec:transellopen}
Atiyah \cite{Atiyah} (see also \cite[Section 3]{ParadanVergne}) defined a distributional index more generally for any element $\alpha_M \in \K^0_G(T_GM)$ where $M$ is a not-necessarily compact Riemannian $G$-manifold.  The construction proceeds as follows.  Atiyah proves \cite[Lemma 3.6]{Atiyah} that one can find a $\bZ_2$-graded Hermitian vector bundle $E=E_0\oplus E_1$ on $M$ and $\sigma_M \in C_b(TM,\pi_{TM}^\ast \End(E))$ an odd, self-adjoint bundle endomorphism  whose restriction to $T_GM$ represents the class $\alpha$, and such that one has $\sigma_M^2=1$ outside $\pi_{TM}^{-1}(K)$ for a $G$-invariant compact subset $K$ of $M$.  Choose a Hermitian vector bundle $F \rightarrow M$ such that $\ti{E}_0=E_0\oplus F$ is trivial, and fix a trivialization.  Let $\ti{E}_1=E_1\oplus F$ and $\ti{\sigma}_M=\sigma_M\oplus \id_F$.  Via $\ti{\sigma}_M$ we obtain a trivialization of $(E_1\oplus F)|_{M\setminus K}$.  Choose a relatively compact $G$-invariant open neighborhood $U$ of $K$, and let $\iota_{U,M}\colon U \hookrightarrow M$ be the inclusion; we will use the same symbol for the induced open inclusion $T_GU\hookrightarrow T_GM$.  The pair $(\ti{E}|_U,\ti{\sigma}_M|_U)$ represents a class $\alpha_U\in \K^0_G(T_GU)$ and $\alpha_M=(\iota_{U,M})_\ast \alpha_U$ by construction.  Choose a $G$-equivariant open embedding $\iota_{U,X}$ of $U$ into a compact $G$-manifold $X$; again we use the same symbol for the induced open inclusion $T_GU\hookrightarrow T_GX$.  Using the trivializations over $U\setminus K$, the bundle $\ti{E}|_U$ and endomorphism $\ti{\sigma}_M|_U$ can be extended trivially to $X$ (denoted $\ti{E}_X$, $\ti{\sigma}_X$ respectively) and represent the class $\alpha_X=(\iota_{U,X})_\ast \alpha_U \in \K^0_G(T_GX)$.  Atiyah defines
\[ \index(\alpha_M)=\index(A_X) \in R^{-\infty}(G) \]
where $A_X$ is any transversally elliptic operator on $X$ such that the (naive) K-theory class of its symbol is $\alpha_X$.  Atiyah proves an excision property  \cite[Theorem 3.7]{Atiyah} showing that the index can be determined just from data on $U$, and hence the construction is independent of the various choices.\\

We can reformulate this construction and Atiyah's excision result in the language of Theorem \ref{thm:KaspInd}: suppose that one manages to choose $\sigma_M$ such that, in addition to the conditions above, one has $(1-\sigma_M^2)\in \S_\Gamma(M)$.  Then $\sigma_M$ determines a class $[\sigma_{M,\tn{c}}]=[(\S_\Gamma(E),\sigma_M)]\in \KK^G(\bC,\S_\Gamma(M))$ refining the class $\alpha_M$.  The subscript `c' is to emphasize that this is a K-theory class whose support is compact over $M$, in contrast with the symbols defining elements of the group $\R\KK^G(M;C_0(M),\S_\Gamma(M))$ that were considered in Section \ref{sec:symbalg}.  One then obtains similar classes $[\ti{\sigma}_{U,\tn{c}}]=[(\S_\Gamma(\ti{E}|_U),\ti{\sigma}|_U)]\in \KK^G(\bC,\S_\Gamma(U))$ refining $\alpha_U$, $[\ti{\sigma}_{X,\tn{c}}]\in [(\S_\Gamma(\ti{E}_X),\ti{\sigma}_X)]\in \KK^G(\bC,\S_\Gamma(X))$ refining $\alpha_X$, and moreover
\begin{equation} 
\label{eqn:functorinclusion}
[\ti{\sigma}_{X,\tn{c}}]=(\iota_{U,X})_\ast[\ti{\sigma}_{U,\tn{c}}], \qquad [\sigma_{M,\tn{c}}]=(\iota_{U,M})_\ast[\ti{\sigma}_{U,\tn{c}}].
\end{equation}
Let $[\ti{\sigma}_{X,\tn{c}}^{\tn{tcl}}]$, $[\ti{\sigma}_{U,\tn{c}}^{\tn{tcl}}]$, $[\sigma_{M,\tn{c}}^{\tn{tcl}}]$ be the corresponding tangential Clifford symbols obtained by KK-product with $\f_{X,\Gamma}$, $\f_{U,\Gamma}$, $\f_{M,\Gamma}$ respectively.  Functoriality of the classes $\f_{-,\Gamma}$ under open embeddings implies the tangential Clifford symbols satisfy analogous formulae to \eqref{eqn:functorinclusion}.\\

Let $p \colon X \rightarrow \pt$ be the collapse map, and $[\sigma_{A,X}]\in \R\KK^G(X;C(X),\S_\Gamma(X))$ the class defined by the symbol of $A_X$, so that $p_\ast[\sigma_{A,X}]=[\ti{\sigma}_{X,\tn{c}}]$.  By Theorem \ref{thm:KaspInd},
\[ \index(A_X)=p_\ast[A_X]=j^G([\ti{\sigma}_{X,\tn{c}}^{\tn{tcl}}])\wh{\otimes}_{G\ltimes \Cl_\Gamma(TX)}[\ol{\partial}^{\cl}_{TX,\Gamma}].\]
Equations \eqref{eqn:functorinclusion}, as well as the functoriality of the KK-product and of the transverse Dolbeault class, give the equivalent formulae
\[ \index(A_X)=j^G([\ti{\sigma}_{U,\tn{c}}^{\tn{tcl}}])\wh{\otimes}_{G\ltimes \Cl_\Gamma(TU)}[\ol{\partial}_{TU,\Gamma}]=j^G([\sigma_{M,\tn{c}}^{\tn{tcl}}])\wh{\otimes}_{G\ltimes \Cl_\Gamma(TM)}[\ol{\partial}^{\cl}_{TM,\Gamma}].\]
We thus obtain the following formula for the index in Atiyah's sense of $\alpha_M=\iota_{T_GM}^\ast[\sigma_{M,\tn{c}}]\in \K^0_G(T_GM)$:
\begin{equation}
\label{eqn:indexopenmfld}
\index(\iota_{T_GM}^\ast[\sigma_{M,\tn{c}}])=j^G([\sigma_{M,\tn{c}}^{\tn{tcl}}])\wh{\otimes}_{G\ltimes \Cl_\Gamma(TM)}[\ol{\partial}^{\cl}_{TM,\Gamma}].
\end{equation}

\subsection{Proof of Theorem \ref{thm:BravInd} (beginning)}\label{sec:BravermanIndex}
The discussion of Sections 4.2--4.5 applies in a more general setting where $E$ is a $\bZ_2$-graded Hermitian vector bundle and $A$ is a transversally elliptic operator. With the aim of proving Theorem \ref{thm:BravInd}, we now return to the setting of interest, where $\nu \colon M \rightarrow \g$ is a taming map with induced vector field $\bm{\!\nu}$ (having a compact vanishing locus), and $\sf{D}_{f\nu}$ is a deformed Dirac operator acting on sections of a $\Cliff(TM)$-module $(E,\c\colon \Cliff(TM)\rightarrow \End(E))$.\\

Using the vector field $\bm{\!\nu}$, define the deformed symbol
\[\sigma_{\!\nu}^0(\xi)=\i\c(\la\xi\ra^{-1}\xi+\bm{\!\nu}), \quad \xi \in TM\simeq T^*M.\]
Since $\bm{\!\nu}$ is a section of $\Gamma$, the support $\supp(\sigma_{\!\nu}^0)\cap T_GM=\{(x,0)\in TM \, : \, \bm{\!\nu}(x)=0\}$.  By assumption the vanishing locus of $\bm{\!\nu}$ is compact, hence the pair $(C_0(T_GM,\pi_{T_GM}^\ast E),\sigma_{\!\nu}^0)$ represents an element $\alpha_{\!\nu}=[\sigma_{\!\nu}^0] \in \K_G^0(T_GM)$, and so has a distributional index. Our goal is to prove Theorem \ref{thm:BravInd}, which states that $\index(\st{D}_{f\!\nu})=\index(\alpha_{\!\nu})$; we will deduce this result as a consequence of the KK-product factorization (Theorem \ref{thm:KK-product}) and Kasparov's index theorem \ref{thm:KaspInd}.\\

As a first step, let us re-write the right-hand-side of Theorem \ref{thm:BravInd} in the language of Section \ref{sec:transellopen}.  Recall that we assumed $|\bm{\!\nu}|\le 1$, with equality outside a $G$-invariant relatively compact open set $U\subset M$.  Define
\begin{equation} 
\label{e:sigmanu}
\sigma_{\!\nu}(\xi)=\i \c\big((1-|\bm{\!\nu}|^2)^{1/2}\la\xi\ra^{-1}\xi+\bm{\!\nu}\big).
\end{equation}
Note that on the open subset of $M$ where $|\bm{\!\nu}|=1$, $\sigma_{\!\nu}(\xi)$ simplifies to the invertible bundle endomorphism $\i \c(\bm{\!\nu})$ (not depending on the fibre variable $\xi$). It follows that $\supp(\sigma_{\!\nu}^0)\cap T_GM=\supp(\sigma_{\!\nu})\cap T_GM$, and the two symbols are homotopic, the formula for the homotopy being given by the same formula as \eqref{e:sigmanu} except with $(1-|\bm{\!\nu}|^2)^{1/2}$ replaced with $(1-|\bm{\!\nu}|^2)^{1/2}t+(1-t)$, where $t \in [0,1]$.  Thus the symbols $\sigma_{\!\nu}$, $\sigma_{\!\nu}^0$ define the same class $\alpha_{\!\nu} \in \K^0_G(T_GM)$.  Since $(1-|\bm{\!\nu}|^2)$ has compact support, one has $(1-\sigma_{\!\nu}^2)\in \S_\Gamma(M)$.  By the discussion in Section \ref{sec:transellopen}, the pair $(\S_\Gamma(E),\sigma_{\!\nu})$ represents a class $[\sigma_{\!\nu,\tn{c}}]\in \KK^G(\bC,\S_\Gamma(M))$ that refines $\alpha_{\!\nu} \in \K^0_G(T_GM)$. By equation \eqref{eqn:indexopenmfld},
\begin{equation}
\label{eqn:RHS}
\index(\alpha_{\!\nu})=j^G([\sigma_{\!\nu,\tn{c}}^{\tn{tcl}}])\wh{\otimes}_{G\ltimes \Cl_\Gamma(TM)}[\ol{\partial}^{\cl}_{TM,\Gamma}].
\end{equation}

The next subsection explains how to factor out $[\!\nu] \in \KK^G(\bC, \Cl_\Gamma(M))$ in the equation above.

\subsection{The symbol class of the transverse Dirac element}
Recall that the same operator $\st{D}$ defines K-homology classes in two different groups, $[\st{D}_M]\in \K^0_G(C_0(M))$, $[\st{D}_{M,\Gamma}]\in \K^0(G\ltimes \Cl_\Gamma(M))$ (Theorem \ref{thm:FundClass}).  The order-$0$ symbol $\sigma(\xi)=\i \c(\la\xi\ra^{-1}\xi)$ (the symbol of $F=\st{D}(1+\st{D}^2)^{-1/2}$) determines an element $[\sigma_M]=[(C_0(TM,\pi_{TM}^\ast E),\sigma)]\in \R\KK^G(M;C_0(M),C_0(TM))$. On the other hand the analogue of $[\st{D}_{M,\Gamma}]$ at the level of symbols is the class $[\sigma_{M,\Gamma}]\in \R\KK^G(M;\Cl_\Gamma(M),\S_\Gamma(M))$ defined by the pair $(\S_\Gamma(E),\sigma)$.  Recall that the Hilbert $\S_\Gamma(M)$-module $\S_\Gamma(E)$ was defined in Definition \ref{def:symbalg} for an arbitrary $G$-equivariant $\bZ_2$-graded Hermitian vector bundle $E$. In our setting $E$ is also a (left) $\Cliff(TM)$-module bundle, and with this additional structure $\S_\Gamma(E)$ also becomes a (left) $\Cl_\Gamma(M)$-module.   
\begin{lemma}
The pair $(\S_\Gamma(E),\sigma)$ represents a class $[\sigma_{M,\Gamma}]\in \R\KK^G(M;\Cl_\Gamma(M),\S_\Gamma(M))$.
\end{lemma}
\begin{proof}
The only property that needs to be checked is that graded commutators $[\c(a),\sigma]$, for $a \in \Cl_\Gamma(M)$, lie in $\mathcal{K}_{\S_\Gamma(M)}(\S_\Gamma(E))$.  It suffices to consider $a \in C_c^\infty(M,\Gamma)$.  Then
\begin{equation} 
\label{eqn:comwithsymb}
[\i\c(a),\sigma(\xi)]=2g(a,\xi)\la \xi \ra^{-1}.
\end{equation}
Since $a$ is a section of $\Gamma$, this vanishes identically for $\xi \in T_GM$, so a fortiori its restriction to $T_GM$ lies in $C_0(T_GM)$.  Its differential in the $\xi$-direction satisfies the required estimate, so \eqref{eqn:comwithsymb} is an element of $\S_\Gamma(M)\subset \mathcal{K}_{\S_\Gamma(M)}(\S_\Gamma(E))$.
\end{proof}

The next lemma provides a symbol analogue of the factorization $[\st{D}_{f\!\nu}]=j^G([\!\nu])\wh{\otimes}_{G\ltimes \Cl_\Gamma(M)}[\st{D}_{M,\Gamma}]$ from Theorem \ref{thm:KK-product}. Recall that the class $[\sigma_{\!\nu,\tn{c}}] \in \KK^G(\bC,\S_\Gamma(M))$ was defined in Section \ref{sec:BravermanIndex} (just before equation \eqref{eqn:RHS}).
\begin{lemma}
\label{lem:factorsymb}
$[\sigma_{\!\nu,\tn{c}}]=[\!\nu]\wh{\otimes}_{\Cl_\Gamma(M)}[\sigma_{M,\Gamma}] \in \KK^G(\bC,\S_\Gamma(M))$.
\end{lemma}
\begin{proof}
Equation \eqref{e:sigmanu} defining $\sigma_\nu$ is a classical KK-product formula of Kasparov (cf. \cite[Proposition 18.10.1]{Blackadar}), applied to the classes $[\!\nu]$ and $[\sigma_{M,\Gamma}]$.
\end{proof}

\ignore{\begin{proof}
Let $F_1=\i\c(\bm{\!\nu})$ acting on $\Cl_\Gamma(M)$, and let $F_2=\sigma(\xi)$ acting on $\S_\Gamma(E)$.  Note that $\Cl_\Gamma(M)\wh{\otimes}_{\Cl_\Gamma(M)}\S_\Gamma(E)\simeq \S_\Gamma(E)$.  Under this identification, $\sigma$ can be viewed as an operator $F_2^\prime$ on $\Cl_\Gamma(M)\wh{\otimes}_{\Cl_\Gamma(M)}\S_\Gamma(E)$.  The calculation \eqref{eqn:comwithsymb} shows that $F_2^\prime$ is an $F_2$-connection (in the sense of \cite{Skandalis}).  Recall also that we assumed $|\bm{\!\nu}|\le 1$.  Under these conditions there is a formula for the operator of the KK product (cf. \cite[Proposition 18.10.1]{Blackadar}),
\[ F=F_1\wh{\otimes} 1+\big((1-F_1^2)^{1/2}\wh{\otimes} 1)F_2^\prime, \]
and in our case this coincides with the operator $\sigma_{\!\nu}$.
\end{proof}}

Applying the lemma to equation \eqref{eqn:RHS} we obtain
\begin{equation}
\label{eqn:RHS2}
\index(\alpha_{\!\nu})=j^G([\!\nu]\wh{\otimes}_{\Cl_\Gamma(M)}[\sigma_{M,\Gamma}^{\tn{tcl}}])\wh{\otimes}_{G\ltimes \Cl_\Gamma(TM)}[\ol{\partial}^{\cl}_{TM,\Gamma}].
\end{equation}

\subsection{End of the proof of Theorem \ref{thm:BravInd}.}
By Theorem \ref{thm:KK-product} and equation \eqref{eqn:DiracDol}, one has
\[ [\st{D}_{f\!\nu}]=j^G\big([\!\nu]\wh{\otimes}_{\Cl_\Gamma(M)}([\sigma_M]\wh{\otimes}_{M}1_{\Cl_\Gamma(M)})\big)\wh{\otimes}_{G\ltimes \Cl_\Gamma(TM)}[\ol{\partial}^{\cl}_{TM,\Gamma}].\]
Comparing this to equation \eqref{eqn:RHS2}, we see that the proof of Theorem \ref{thm:BravInd} is completed by the following result, which is the symbol analogue of Theorem \ref{thm:Significance}.
\begin{proposition} \label{thm:rotation}
$[\sigma_M]\wh{\otimes}_{C_0(M)}1_{\Cl_\Gamma(M)}=[\sigma_{M,\Gamma}^{\tn{tcl}}] \in \KK^G(\Cl_\Gamma(M),\Cl_\Gamma(TM))$.
\end{proposition}
\begin{proof}
Note that
\[ \mathcal{H}:=\S_\Gamma(E)\wh{\otimes}_{\S_\Gamma(M)}\Cl_\Gamma(TM)\simeq C_0(TM,\pi_{TM}^\ast E)\wh{\otimes}_{C_0(TM)}\Cl_\Gamma(TM)\simeq C_0(TM,\pi_{TM}^\ast E)\wh{\otimes}_{C_0(M)}\Cl_\Gamma(M) \]
as Hilbert $\Cl_\Gamma(TM)=C_0(TM)\wh{\otimes}_{C_0(M)}\Cl_\Gamma(M)$-modules; thus the KK-elements on the left and right hand sides are naturally represented on the same $\Cl_\Gamma(TM)$-module $\mathcal{H}$.  The representations of $\Cl_\Gamma(M)$ differ however; we denote the representation for $[\sigma_{M,\Gamma}^{\tn{tcl}}]$ (resp. $[\sigma_M]\wh{\otimes}_{C_0(M)}1_{\Cl_\Gamma(M)}$) by $\pi_0$ (resp. $\pi_1$), where for $a \in \Cl_\Gamma(M)$,
\[ \pi_0(a)=\c(a)\wh{\otimes}1, \qquad \pi_1(a)=1\wh{\otimes}a \]
(here $1\wh{\otimes}a$ denotes the operator $e\wh{\otimes}f\mapsto (-1)^{\tn{deg}(e)\tn{deg}(a)}e\wh{\otimes}af$).\\

The operator representing $[\sigma_M]\wh{\otimes}_{M}1_{\Cl_\Gamma(M)}$ is $\sigma(m,\xi)\wh{\otimes}1=\i\c(\la\xi\ra^{-1}\xi)\wh{\otimes}1$.  The operator representing the product
\[ [\sigma_{M,\Gamma}^{\tn{tcl}}]=[\sigma_{M,\Gamma}]\wh{\otimes}_{\S_\Gamma(M)}[\f_{M,\Gamma}]\in \KK^G(\Cl_\Gamma(M),\Cl_\Gamma(TM)),\]
can be taken to be the same as that in \eqref{eqn:tcl}, namely
\[ S_0=N_1^{1/2}(\sigma\wh{\otimes}1)+N_2^{1/2}(1\wh{\otimes}\f_{M,\Gamma})\]
where the weights $N_2=1-N_1 \in C_b(TM)$ are as in equation \eqref{eqn:weights}; indeed the only additional condition that needs to be checked is the compactness of the commutators $[\pi_0(a),S_0]$, and this follows from the observation that for $a \in C_0(M,\Gamma)$ one has $g(a,\la\xi\ra^{-1}\xi)\in \S_\Gamma(M)$, together with Lemma \ref{lem:decayN1}.\\

We perform a `rotation' homotopy simultaneously on the operator $S_0$ and representation $\pi_0$.  For $t \in [0,1]$ let
\[ \pi_t(a)=\cos(\tfrac{\pi}{2} t)\c(a)\wh{\otimes}1+\sin(\tfrac{\pi}{2}t)1\wh{\otimes}a, \qquad S_t=N_1^{1/2}(\sigma\wh{\otimes}1)+N_2^{1/2}\f_{M,\Gamma,t}, \]
where
\[ \f_{M,\Gamma,t}(m,\xi)=\sin(\tfrac{\pi}{2}t)\i\c(\la\varphi_m(\xi)\ra^{-1}\varphi_m(\xi))\wh{\otimes}1-\cos(\tfrac{\pi}{2}t)1\wh{\otimes}\i\la\varphi_m(\xi)\ra^{-1}\varphi_m(\xi).\]
It is clear that $\pi_0$, $\pi_1$, $S_0$ coincide with the previous definitions.  Let us check that this is a homotopy of Kasparov cycles.  The commutator condition for the representation follows because $[\pi_t(a),\f_{M,\Gamma,t}]=0$ for all $a \in \Cl_\Gamma(M)$ and $t \in [0,1]$.  For $f \in C_0(M)$, the function $f(1-S_t^2)$ is the same as $f(1-S_0^2)$ except for an additional cross-term
\begin{equation} 
\label{eqn:crosstermsymbol}
-2\sin(\tfrac{\pi}{2}t)N_1^{1/2}(m,\xi)\cdot N_2^{1/2}(m,\xi)\cdot f(m)\cdot g\big(\tfrac{\xi}{\la\xi\ra},\tfrac{\varphi_m(\xi)}{\la\varphi_m(\xi)\ra}\big). 
\end{equation}
The product $f(m)\cdot g(\la\xi\ra^{-1}\xi,\la\varphi_m(\xi)\ra^{-1}\varphi_m(\xi)) \in \S_\Gamma(M)$.  Since $N_2 \le 1$, Lemma \ref{lem:decayN1} implies that \eqref{eqn:crosstermsymbol} lies in $C_0(TM)$.  From this it follows that $\pi_t(a)(1-S_t^2) \in \mathcal{K}_{\Cl_\Gamma(TM)}(\mathcal{H})$ for all $a \in \Cl_\Gamma(M)$.\\

After the homotopy, the representations of $\Cl_\Gamma(M)$ on $\mathcal{H}$ for the two cycles agree, and we are left with the operator
\[ S_1=N_1^{1/2}(\sigma\wh{\otimes}1)+N_2^{1/2}\f_{M,\Gamma,1}, \qquad \f_{M,\Gamma,1}(m,\xi)=\i\c(\la\varphi_m(\xi)\ra^{-1}\varphi_m(\xi))\wh{\otimes}1.\]
Note that the graded commutator $[\sigma\wh{\otimes}1,\f_{M,\Gamma,1}]$ is the function
\[ \tfrac{2}{\la\xi\ra \la\varphi_m(\xi)\ra} g(\xi,\varphi_m(\xi))\]
and $g(\xi,\varphi_m(\xi))=g(\xi,\rho_m \rho_m^\top(\xi))\ge 0$.  It follows that the operator
\[ [\sigma\wh{\otimes}1,S_1] \]
is positive (and a fortiori positive modulo compacts).  By a well-known criterion of Connes-Skandalis (cf. \cite[Proposition 17.2.7]{Blackadar}), the cycles $(\mathcal{H},\pi_1,S_1)$, $(\mathcal{H},\pi_1,\sigma\wh{\otimes}1)$ are operator homotopic.
\end{proof}

\appendix

\section{The case of non-complete manifolds} \label{AppendixA}

This appendix follows up Section \ref{Section1}, and uses the same notation. Recall that on non-complete manifolds, the main issue comes from the possible non-self-adjointness of the Dirac operator, so that $K$-homology classes have to be constructed with slightly more care. Adapting the techniques given in \cite{Higson} or \cite[Chapter 10]{HigsonRoe}, we generalize the construction of the class $[\st{D}_{M,\Gamma}] \in K^0(G \ltimes \Cl_{\Gamma}(M))$ to the case where $M$ is not complete.  Throughout this section $\sim$ stands for equality modulo compact operators.\\

Let $\chi: \bR \to [-1,1]$ be a `normalizing function', i.e a continuous odd function which is positive on $(0,\infty)$ and tends to $1$ at $ \infty$, and let $H=L^2(M,E)$. Cover $M$ with relatively compact $G$-invariant open sets $U_j$ and let $f_j^2$ be a $G$-invariant partition of unity subordinate to the cover.  Let $\st{D}_j$ be a $G$-equivariant essentially self-adjoint operator agreeing with $\st{D}$ on $U_j$ (for example, compress $\st{D}$ between suitable $G$-invariant bump functions with support contained in a compact neighborhood of $U_j$).  Let
\[ F=\sum_j f_j \chi(\st{D}_j)f_j \]
which converges in the strong operator topology to a bounded self-adjoint operator. 

\begin{lemma}[\cite{HigsonRoe} Lemma 10.8.3]
\label{lem:EqualModCompacts}
Let $\st{D}_1$, $\st{D}_2$ be essentially self-adjoint first order differential operators on $M$ which restrict to the same elliptic operator on some open subset $U \subset M$.  Let $g \in C_0(U)$.  Then $\chi(\st{D}_1)g \sim \chi(\st{D}_2)g$.
\end{lemma}

Now, let $a=h\wh{\otimes} \alpha \in G \ltimes \Cl_{\Gamma}(M)$ where $h \in C^\infty(G)$ and $\alpha \in \Cl^\infty_{\Gamma,c}(M)$.  Choose a $G$-invariant compactly supported cut-off function $f$ equal to $1$ on the support of $\alpha$, and let $\st{D}_f$ be an essentially self-adjoint operator that agrees with $\st{D}$ in a neighborhood of the support of $f$. Then, the lemma above (combined with the $G$-invariance of $f$) shows that
\[ [F,a] \sim [\chi(D_f), a] \, ; \quad a(F^2-1) \sim a(\chi(D_f)^2 - 1). \]
Following the proof of Theorem \ref{thm:FundClass} in the complete case, the operators on the right hand sides are compact, so that $(H,F)$ is a Fredholm module. \\

Moreover, if $F'$ is an operator constructed the same way as $F$ but from a different partition of unity, Lemma \ref{lem:EqualModCompacts} shows for every $a \in G\ltimes \Cl_\Gamma(U)$,
\[ a(FF'+F'F)a^\ast \sim 2a\chi(\st{D}_f)^2a^\ast \ge 0 \tn{ modulo compact operators} \]
($f$ being a function depending on $a$ as above), which is a well-known sufficient condition for $F'$ to be norm-continuously homotopic to $F$ (see \cite{Skandalis}). Therefore, the $K$-homology class $[(H,F)] \in K^0(G \ltimes \Cl_{\Gamma}(M))$ does not depend on the choice of the partition of unity (and the cover). Finally, if $M$ is complete, a similar calculation shows that $[(H,F)] = [(H, \chi(\st{D}))]$ in $K^0(G \ltimes \Cl_{\Gamma}(M))$. 

\section{Proof of Proposition \ref{prop:OpenRestrict}} \label{AppendixB}

The proof is standard and follows closely \cite[Proposition 10.8.8]{HigsonRoe}. We include it for the convenience of the reader. Let $U$ be a $G$-invariant open set of $M$, and $\iota_U^\ast \colon \K^{0}(G\ltimes \Cl_\Gamma(M)) \rightarrow \K^{0}(G\ltimes \Cl_\Gamma(U))$ be the associated extension-by-$0$ homomorphism. Recall we want to prove that $ \iota_U^\ast[\st{D}_{M,\Gamma}]=[\st{D}_{U,\Gamma}]$.

\begin{proof}
Let $(H, F) := (L^2(M,E), F = \st{D}(1+\st{D}^2)^{-\frac{1}{2}})$ be the Fredholm module of Theorem \ref{thm:FundClass}, and let $P$ denote the orthogonal projection $H \rightarrow H_U=L^2(U,E)$ (given by multiplication by the characteristic function of the subset $U$). Then
\[ PFP \colon H_U \rightarrow H_U \]
is a bounded operator, and $(H_U, PFP)$ is a Fredholm module over $G\ltimes \Cl_\Gamma(U)$ (to see this, note that $P$ commutes with $G\ltimes \Cl_\Gamma(U)$, and $P|_{H_U}=1$). \\


Let $Q$ be the orthogonal projection to $L^2(M \setminus U,E|_{M \setminus U})$, i.e $Q=1-P$.  In terms of the decomposition $H=PH\oplus QH$, the operator $F$ becomes the $2 \times 2$ matrix:
\[ F=\left(\begin{array}{cc} PFP & PFQ\\QFP & QFQ \end{array}\right).\]
Notice that for $a \in G\ltimes \Cl_\Gamma(U)$, $aQ=0$.  Moreover (recall $\sim$ stands for equality up to compact operators)
\[ aPFQ=PaFQ \sim PFaQ=0.\]
Consequently
\[ aF \sim \left(\begin{array}{cc} aPFP & 0\\0 & 0 \end{array}\right).\]
This shows that the restriction of the $G\ltimes \Cl_\Gamma(M)$-Fredholm module $(H,F)$ to a $G\ltimes \Cl_\Gamma(U)$-Fredholm module equals $(H_U,PFP)$ up to a locally compact perturbation (the entries $QFP$, $PFQ$ and $QFQ$ in the matrix for $F$) and a degenerate module (namely $(QH,0)$).  Thus
\[ \iota_U^\ast[(H,F)]=[(H_U,PFP)]. \]
It remains to check that the cycle $(H_U,PFP)$ for $K^0(G\ltimes \Cl_\Gamma(U))$ is operator homotopic to $(H_U,F_U)$ where $F_U=\Sigma_j f_j \chi(\st{D}_j)f_j$ is the operator constructed in Appendix \ref{AppendixA}.  Let $a=h\wh{\otimes} \alpha$, $h \in C^\infty(G)$, $\alpha \in \Cl_{\Gamma,c}^\infty(U)$. Fix $j$ and consider 
\begin{equation} 
\label{eqn:aligned}
a^\ast(PFPf_j\chi(\st{D}_j)f_j+f_j\chi(\st{D}_j)f_jPFP)a 
\end{equation}
as an operator on $H_U$.  Note that $Pa=a$ since $\alpha$ has support contained in the $G$-invariant set $U$.  Thus
\[ PFPf_j\chi(\st{D}_j)f_ja\sim PFPaf_j\chi(\st{D}_j)f_j=PFaf_j\chi(\st{D}_j)f_j\sim PFf_j\chi(\st{D}_j)f_ja. \]
Applying similar arguments to the other factors of $P$ in \eqref{eqn:aligned}, it follows that, modulo compact operators, the operator in \eqref{eqn:aligned} is
\[ a^\ast (Ff_j\chi(\st{D}_j)f_j+f_j\chi(\st{D}_j)f_jF)a \]
and the latter is positive modulo compact operators, by the results of Appendix \ref{AppendixA} applied to the operator $F$ on $M$.  We obtain that the operator in \eqref{eqn:aligned} is positive modulo compact operators. Since $af_j$ vanishes for all but finitely many $j$, we conclude that
\[ a^\ast (PFP F_U+F_U PFP)a \ge 0 \quad \text{mod }\mathcal{K}(H_U). \]
This proves $(H_U,PFP)$ is homotopic to the cycle $(H_U,F_U)$ from Appendix \ref{AppendixA}.
\end{proof}

\section{Proof of Proposition \ref{prop:DiracBoundary}} \label{AppendixC}

Again, the proof is standard and follows closely \cite[Proposition 11.2.15]{HigsonRoe} or \cite{BDTRelativeCycles}. It is included for the convenience of the reader. \\

Recall the context: $M=\partial{\widetilde{M}}$ is the boundary of a Riemannian $G$-manifold $\widetilde{M}$, and let $W=\widetilde{M} \smallsetminus M$. Consider the $C^\ast$-algebra extension:
\begin{equation*}
0\rightarrow G\ltimes \Cl_\Gamma(W) \rightarrow G\ltimes \Cl_\Gamma(\widetilde{M}) \rightarrow G\ltimes \Cl_\Gamma(M) \rightarrow 0.
\end{equation*}  
and the corresponding boundary map $\partial$ in $K$-homology. Let $\ti{E}\rightarrow \ti{M}$ be an ungraded Clifford module bundle, $\ti{\st{D}}$ a Dirac operator acting on sections of $\ti{E}$, and $[\ti{\st{D}}_{W,\Gamma}] \in K^1(G\ltimes \Cl_\Gamma(W))$ the corresponding K-homology class. The restriction to the boundary $E=\ti{E}|_{\partial \ti{M}}$ becomes a $\bZ_2$-graded $\Cliff(TM)$-module bundle with the graded subbundles $E^{\pm}$ being the $\pm \i$-eigenbundles of $\c(n)$, where $n$ is the inward unit normal vector to the boundary.  Let $\st{D}$ be a Dirac operator acting on sections of $E$ and $[\st{D}_{M,\Gamma}] \in K^0(G\ltimes \Cl_\Gamma(M))$ the corresponding K-homology class. We want to show that $\partial[\widetilde{\st{D}}_{W,\Gamma}]=[\st{D}_{M,\Gamma}]$.

\begin{proof}
Let $\varepsilon > 0$ such that $(0,\varepsilon) \times M \subset W$ be a collar neighborhood of $M$ for which $G$ acts trivially on the $(0, \varepsilon)$ part. We then have the following morphisms of extensions 
\[
\xymatrix{ 
0 \ar[r] & G\ltimes \Cl_\Gamma(W) \ar[r]  & G\ltimes \Cl_\Gamma(\widetilde{M}) \ar[r]  & G\ltimes \Cl_\Gamma(M) \ar[r] \ar@{=}[d] & 0 \\
0 \ar[r] & G\ltimes \Cl_\Gamma((0, \varepsilon) \times M) \ar[r] \ar[u]_{\text{extension-by-0}} & G\ltimes \Cl_\Gamma([0, \varepsilon) \times M) \ar[r] \ar[u]_{\text{extension-by-0}} & G\ltimes \Cl_\Gamma(M) \ar[r] \ar@{=}[d] & 0 \\
0 \ar[r] & C_0(0, \varepsilon) \, \wh{\otimes} \, \big(G\ltimes \Cl_\Gamma(M)\big) \ar[u]_\simeq \ar[r] & C_0[0, \varepsilon) \, \wh{\otimes} \, \big(G\ltimes \Cl_\Gamma(M)\big) \ar[u]_\simeq \ar[r] & G\ltimes \Cl_\Gamma(M) \ar[r] & 0 }
\]
Notice that the bottom extension is simply the cone extension, so that the associated boundary map is the suspension isomorphism
\[ \delta : \K^{\bullet + 1}\big(C_0(0, \varepsilon) \wh{\otimes} (G\ltimes \Cl_\Gamma(M))\big) \to \K^{\bullet}(G\ltimes \Cl_\Gamma(M)). \] 
Since the class $[\ti{\st{D}}_{W,\Gamma}]$ does not depend on the choice of the metric, we can equip $(0, \varepsilon) \times M$ with the product metric. This way, the class $[\widetilde{\st{D}}_{W,\Gamma}] \in \K^{1}(G\ltimes \Cl_\Gamma(W))$ identifies over the collar neighborhood with the exterior $\KK$-product $[\st{D}_{(0,\varepsilon)}] \wh{\otimes} [\st{D}_{M,\Gamma}]$, where $\st{D}_{(0,\varepsilon)}$ is the Dirac operator on $(0,\varepsilon)$ and $[\st{D}_{(0,\varepsilon)}]\in \K^1(C_0(0,\varepsilon))$. But the map $[\st{D}_{(0,\varepsilon)}] \wh{\otimes} \, \bm{.}$ is inverse to the suspension isomorphism, so that
\[ \delta([\st{D}_{(0,\varepsilon)}] \wh{\otimes} [\st{D}_{M,\Gamma}])=[\st{D}_{M,\Gamma}]\]
The conclusion then follows from the naturality of the boundary map.
\end{proof}


\providecommand{\bysame}{\leavevmode\hbox to3em{\hrulefill}\thinspace}
\providecommand{\MR}{\relax\ifhmode\unskip\space\fi MR }
\providecommand{\MRhref}[2]{%
  \href{http://www.ams.org/mathscinet-getitem?mr=#1}{#2}
}
\providecommand{\href}[2]{#2}

\end{document}